\documentclass[11pt, reqno]{amsart}

\usepackage{amsfonts,amssymb,amsmath,amsthm, ulem}
\usepackage[headinclude,DIV13]{typearea}
\usepackage[utf8]{inputenc}
\usepackage{hyperref}
\usepackage{geometry}
\usepackage{graphicx, psfrag,enumerate,enumitem}

\usepackage{mathrsfs}  

\usepackage{soul}
\usepackage{cancel}
\usepackage{color}
 \usepackage{setspace}

\newcommand{\al}[1]{{\color{black} #1}} 
\newcommand{\cha}[1]{{\color{black} #1}} 

\newtheorem{thm}{Theorem}[section]
\newtheorem{prop}[thm]{Proposition}

\newtheorem{cor}[thm]{Corollary}
\newtheorem{lemma}[thm]{Lemma}

\newtheorem{ass}{Assumption}

\newtheorem{rem}[thm]{Remark}
\numberwithin{equation}{section}

\newtheorem{innercustomass}{Assumption}
\newenvironment{customass}[1]
  {\renewcommand\theinnercustomass{#1}\innercustomass}
  {\endinnercustomass}
  
\theoremstyle{definition} 
\theoremstyle{definition}

\renewcommand{\P}{\mathbb{P}}
\newcommand{\R}{\mathbb{R}}

\newcommand{\E}{\mathbb{E}}

\newcommand{\N}{\mathbb{N}}

\newcommand{\eps}{\varepsilon}

\newcommand{\brho}{\mathfrak b}
\newcommand{\kapparq}{\mathfrak{d}}
\newcommand{\m}{\mathfrak m}

\newcommand{\ud}{\mathrm{d}}
\newcommand{\laplexp}{\phi}
\newcommand{\mintheta}{\vartheta}

\DeclareMathOperator{\argmin}{argmin}

\numberwithin{equation}{section}

\begin{document}

\title[Parasite infection in a cell population]
{Parasite infection in a cell population: role of the partitioning kernel}

\author{Aline Marguet}
\address{Aline Marguet, Univ. Grenoble Alpes, INRIA, 38000 Grenoble, France}
\email{aline.marguet@inria.fr}

\author{Charline Smadi}
\address{Charline Smadi, Univ. Grenoble Alpes, INRAE, LESSEM, 38000 Grenoble, France
 and Univ. Grenoble Alpes, CNRS, Institut Fourier, 38000 Grenoble, France}
\email{charline.smadi@inrae.fr}

\date{}

\begin{abstract}
We consider a cell population subject to a parasite infection.
Cells divide at a constant rate and, at division, share the
parasites they contain between their two daughter cells. The
sharing may be asymmetric, and its law may depend on the
number of parasites in the mother. Cells die at a rate which may depend on the number of parasites they carry, and are also killed when this number explodes. We study the survival of the cell population as well as the mean number of parasites in the cells, and focus on the role of the parasites partitioning kernel at division.
 \end{abstract}
 
\maketitle

 \vspace{0.2in}


\bigskip


 \vspace{0.5cm}

  \tableofcontents

\section*{Introduction}

We are interested in the modelling of a parasite infection in a cell population, and, in particular, in the role of the stochasticity of repartition of the parasites at cell division. From the pioneering work of Kimmel \cite{Kimmel}, several models and associated analysis have been proposed, both in discrete \cite{bansaye2008, bansaye2009, alsmeyer2013, alsmeyer2016} and continuous time \cite{BT11,BPS}. The asymmetric repartition of parasites is taken into account in all of those work: in \cite{bansaye2008, osorio2020level}, each parasite chooses to go to one daughter cell with probability $p$ (and to the other with probability $1-p$), and in \cite{bansaye2009}, a random environment is considered (the probability generating functions of the number of parasites at birth in the two daughter-cells of each cell in the population are i.i.d random variables). In branching-within-branching models, independently for each parasite and with the same distribution, the descendants are shared between the daughter cells. A different approach has been proposed in \cite{BT11}, removing the independence property of the sharing of parasites descending from different lineages. Following the dynamics of the ($\R$-valued) number of parasites inside the cells (rather than a discrete count), this model assumes that a cell with $x$ parasites, splits into two daughter cells with a number of parasites $\Theta x$ and $(1-\Theta)x$ respectively, with $\Theta$ a random variable on $[0,1]$. Here, we extend this approach and explore the role of the random variable $\Theta$ in the proliferation of the infection. We assume that cells divide at a constant rate. Their death rate may depend on the number of parasites they contain and they may additionally be killed when this number explodes. 
The dynamics of the number of parasites in a cell is given by a Stochastic Differential Equation (SDE) with drift, diffusion and positive jumps.
At division the parasites of a
cell are shared between its two daughters according to a
partitioning kernel which may depend on the number of
parasites. 
\al{Similar to the works} \cite{Kimmel,bansaye2008,bansaye2009,BT11,BPS,alsmeyer2013,alsmeyer2016,osorio2020level}, we are interested in the long time behaviour of the parasite infection. More precisely, we will study the number of cells alive, as well as the number of parasites in the cells at large time.
This work complements \cite{marguet2022spread}, where we considered division rates which could depend on the number of parasites contained in the cells but fixed partitioning kernels at division. \al{ Note that \cite{marguet2022spread} and this paper comes from the split of an earlier draft \cite{marguet2020parasite}, with additional results on the study of the effects of the partitioning on the fate of the cell population.} Assuming here that the division rate is constant allows 
us to consider partitioning kernels depending on the number of parasites in the cell at division, and to focus on the effect of the partitioning kernel on the long time behaviour of the infection. In particular, we compare partitioning strategies and show that a symmetric division (half of the
parasites in each daughter cell) is the worst choice for the cell population in terms of survival.
We give quantitative conditions on the level of infection for the cell population to survive, for both uniform and equal sharing partitioning kernels. We also explore numerically the difference between deterministic and a class of random partitioning laws, \al{highlighting the fact that randomness and asymmetry seem to be the keys to explain survival. Then, we prove that any partitioning kernel is better for survival than its deterministic counterpart with the same expected minimum value.}  Finally, we prove that whatever the growth of the parasites, there exist partitioning kernels enabling the cell population to survive the infection.

Our proof strategy consists in introducing a spinal decomposition.
It amounts to distinguishing a particular line of descent in the population, constructed from a size-biased tree \cite{LPP}, and to prove that the dynamics of the trait along this particular lineage is representative of the dynamics of the trait of a typical individual in the population, {\it i.e.} an individual picked uniformly at random. We refer to \cite{georgii2003supercritical,hardy2009spine,bansaye2011limit,cloez2017limit,marguet2016uniform, marguet2017law} for general results on these topics in the continuous-time setting.

The paper is structured as follows. In Section \ref{section_model}, we define the population process and give assumptions ensuring its existence and uniqueness as the strong solution to a SDE. Sections \ref{sec_mean_numb_cells} and \ref{sec:role-partitioning} are dedicated to the study of the asymptotic behaviour of the mean number of cells alive in the population for various dynamics for the parasites. In particular, we compare different strategies for the sharing of the parasites at division and give explicit conditions ensuring extinction or survival of the cell population. In Section \ref{section_beta}, we focus on 
the case of a parasites dynamics without stable positive jumps and study the asymptotic behaviour of the proportion of infected cells. 
Sections \ref{sec:MTO} and \ref{sec:proofs} are dedicated to the proofs.
\\

In the sequel $\N:=\{0,1,2,...\}$ will denote the set of nonnegative integers, $\R_+:=[0,\infty)$ the real line,  $\overline{\R}_+:=\R_+\cup \{ + \infty \}$,
and $\R_+^*:=(0,\infty)$.
We will denote by $\mathcal{C}_{0}^2(\R_+)$ the set of twice continuously differentiable functions on $\R_+$ vanishing at $0$ and infinity. Finally, for any stochastic 
process $X$ on $\overline{\mathbb{R}}_+$ or $Z$ on the set of point measures on $\overline{\mathbb{R}}_+$, we \al{use 
$\mathbb{E}_x\left[f(X_t)\right]$ and $\mathbb{E}_{\delta_x}\left[f(Z_t)\right]$ as shorthand for $\E\left[f(X_t)\big|X_0 = x\right]$ and $\E\left[f(Z_t)\big|Z_0 = \delta_x\right]$ respectively}.

\section{Definition of the population process} \label{section_model}

\subsection{Parasites dynamics in a cell}

Each cell contains parasites whose quantity, \al{denoted by $\mathfrak{X}_t$,} evolves as a diffusion with positive jumps. More precisely, we consider the SDE
\begin{align} \nonumber \label{X_sans_sauts} \mathfrak{X}_t =x + \int_0^t g(\mathfrak{X}_s)ds
+\int_0^t\sqrt{2\sigma^2 (\mathfrak{X}_s)}dB_s &+
\int_0^t\int_0^{p(\mathfrak{X}_{s^-})}\int_{\mathbb{R}_+}z\widetilde{Q}(ds,dx,dz)\\&+
\int_0^t\int_0^{\mathfrak{X}_{s^-}}\int_{\mathbb{R}_+}zR(ds,dx,dz),
\end{align}
where $x$ is nonnegative, $g$, $\sigma \geq 0$ and $p\geq0$ are real functions on $\overline{\mathbb{R}}_+$, $B$ is a standard Brownian motion,  
$\widetilde{Q}$ is a compensated Poisson point measure (PPM) with intensity 
$ds\otimes dx\otimes \pi(dz)$, $\pi$ is a nonnegative measure on $\mathbb{R}_+$,
$R$ is a PPM with intensity 
$ds\otimes dx\otimes \rho(dz)$, with
\begin{align*} 
%\label{def_rho}
\rho(dz)=\frac{c_\brho\brho(\brho+1)}{\Gamma(1-\brho)}\frac{1}{z^{2+\brho}}\ud z, \quad z \in \R_+
\end{align*}
where $\brho \in (-1,0)$ and $c_\brho<0$ (see \cite[Section 1.2.6]{kyprianou2006introductory} for details on stable distributions and processes).
Finally, $B$, $Q$ and $R$ are independent. \al{The function $g$ describes the deterministic part of the growth of the number of parasites. In particular, $g(x)=gx$, for some $g>0$, corresponds to an exponential growth. The diffusion term describes the demographic stochasticity of the parasites. Finally, the last two integrals correspond to two different type of jumps in the dynamic of the number of parasites, describing possible burst of parasites: jumps of finite size and jumps of possibly infinite size.}

\al{We will later provide conditions} under which the SDE \eqref{X_sans_sauts} has a unique nonnegative strong solution.
\al{Under these conditions}, it is a Markov process with infinitesimal generator $\mathcal{G}$, satisfying for all $f\in C_{0}^2(\mathbb{R}_+)$,
\begin{align} \label{def_gene}
\mathcal{G}f(x) = g(x)f'(x)+\sigma^2(x)f''(x)&+p(x)\int_{\mathbb{R}_+}\left(f(x+z)-f(x)-zf'(x)\right)\pi(dz)\\&+x\int_{\mathbb{R}_+}\left(f(x+z)-f(x)\right)\rho(dz)\nonumber ,
\end{align}
and $0$ and $+ \infty$ are two absorbing states.
Following \cite{marguet2016uniform}, we denote by $(\Phi(x,s,t),s\leq t)$ the corresponding stochastic flow {\it i.e.} 
the unique strong solution to \eqref{X_sans_sauts} satisfying $\mathfrak{X}_s = x$ and the dynamics of the trait between division events is well-defined.

\subsection{Cell division}

Each cell carrying a number $x$ of parasites divides at rate $r\al{>0}$ and is replaced by two 
daughter cells with number of parasites $\Theta(x,\zeta) x$ and $(1-\Theta(x,\zeta))x$, where $\Theta(x, \zeta)$ is a symmetric random variable on $(0,1)$, with associated distribution $\kappa(x,\cdot)$,
%distributed as $T(\zeta,x)$,
$\Theta$ is a measurable function from $\mathbb{R}_+\times [0,1]$ to $(0,1)$ and $\zeta$ is a uniform random variable on $[0,1]$. This formalism will prove useful for the
use of Poisson point measures. However, for the sake of simplicity, we will often omit \al{to show} the dependence in $\zeta$ and write $\Theta(x)$ for the random variable corresponding to the proportion of parasites at birth, instead of $\Theta(x,\zeta)$. Finally, we assume that $\sup_{x \in \R_+}|\E[\ln \Theta(x)]|<\infty$.

\subsection{Cell death}

Cells can die because of two mechanisms. 
First, they have a natural death rate $q(x)$ which depends on the number of parasites $x$ they carry. 
Second, cells die when the number of parasites they carry explodes ({\it i.e.}, reaches infinity in finite time), as a proper 
functioning of the cell is not possible anymore. 

\cha{\begin{rem}
To model the second mechanism of death we will use a technical trick consisting in letting cells with an infinite number of parasites exist and reproduce, giving birth to daughter cells with an infinite number of parasites. As it will appear later, this allows us to derive Many-to-one formulas (see Section \ref{sec:MTO}).
\end{rem}}

\subsection{\al{Host-parasite measure-valued process}} \label{host-para-Z}

We use the classical Ulam-Harris-Neveu notation to identify each individual. Let us denote by
\begin{equation*} \label{ulam_not} \mathcal{U}:=\bigcup_{n\in\mathbb{N}}\left\{0,1\right\}^{n}
\end{equation*}
the set of possible labels, $\mathcal{M}_P(\overline{\mathbb{R}}_+)$ the set of point measures on $\overline{\mathbb{R}}_+$, and 
$\mathbb{D}(\mathbb{R}_+,\mathcal{M}_P(\overline{\mathbb{R}}_+))$, the set of c\`adl\`ag measure-valued processes. 
\al{We denote by $Z$ the host-parasite measure-valued process: $Z\in \mathbb{D}(\mathbb{R}_+,\mathcal{M}_P(\overline{\mathbb{R}}_+))$, and for all $t \geq 0$,} 
\begin{equation} \label{Ztdirac}
Z_t = \sum_{u\in V_t}\delta_{X_t^u},
\end{equation}
where $V_t\subset\mathcal{U}$ denotes the set of individuals \cha{in the population} at time $t$ and $X_t^u$ the number of parasites \al{hosted by cell $u$} at time $t$. 
\cha{Recall that if $X_t^u<\infty$ the cell $u$ is alive at time $t$, and if $X_t^u=\infty$ the cell $u$ is dead at time $t$.}
By convention, $Z_t$ is the null measure if $V_t = \emptyset$.
By extension, for $u \in V_t$ and any $s \leq t$, $X_s^u$ denotes the number of parasites in the ancestor of $u$ in the population at time $s$. Thus, $(X_s^u, s \leq \sup(\mathfrak{t}\geq 0 \text{ s.t. } u \neq V_\mathfrak{t}))$ follows \eqref{X_sans_sauts} between events of division impacting the lineage under consideration.

\cha{Under technical assumptions presented in the appendices for the sake of readability (see Assumption \ref{ass_A} in Appendix \ref{append_unic_exi}), we can prove that the host-parasite measure-valued process $Z$ is well-defined as the unique solution of a SDE.}
\al{For the ease of presentation, we make the standing assumption that all appearing
processes satisfy Assumption \ref{ass_A}}.\\

We will now investigate the long time behaviour of the infection in the cell population.
As we explained above, the strategy to obtain information at the population level is to introduce an auxiliary process providing information on the 
behaviour of a `typical individual'. 
We will provide a general expression for this auxiliary process in Section \ref{sec:MTO}.

\section{Mean number of cells alive: General results} \label{sec_mean_numb_cells}

We denote by $\mathfrak{C}_t$ the number of cells alive at time $t$. \al{Recall that a cell can die either by natural death, or if its number of parasites reaches infinity in finite time.}
\cha{As a consequence, $\mathfrak{C}_t$ may be defined as follows:
$$ \mathfrak{C}_t:= \sum_{u\in V_t}\mathbf{1}_{\{X_t^u<\infty\}} .$$}
We give here results on the asymptotic behaviour of $\mathfrak{C}_t$.
For general parasites dynamics, we give sufficient conditions for the cell population to survive with positive probability (see Proposition \ref{CS_ext_ps} below). Moreover, for specific dynamics of the parasites population, we give the asymptotic order of magnitude of the mean number of cells alive in the population. As in \cite[Proposition 2.1]{palau2016asymptotic}, we exhibit three different regimes, depending on the parameters of both the cells and the parasites dynamics (see Proposition \ref{CNS_ext_ps} below).
It allows us to study the effects of parasites growth rate and diffusion parameter, and cells division and death rates. The next section will be devoted to the study of the effect of the partitioning kernel on the average number of cells alive in large time.
\\
 
Let us introduce 
a random variable $\Theta$ on $(0,1)$ with symmetric distribution $\kappa$ satisfying
\begin{equation}\label{ass_Theta}
\E\left[ |\ln\Theta| \right]=\int_0^1 |\ln \theta|\kappa(d\theta)<\infty,
\end{equation}
as well as the function 
\begin{equation}\label{defLaplexp}
\laplexp(\lambda):= \lambda (g-\sigma^2) + \lambda^2 \sigma^2 + 2 r \left(\E[ \Theta^\lambda]-1\right),
\end{equation}
for any $\lambda \in (\lambda^-, \infty)$, where 
\begin{equation*} \lambda^-:= \inf \{\lambda <0: \laplexp(\lambda)<\infty\}.
\end{equation*}
The function $\laplexp$ is the Laplace exponent of a Lévy process (see the proof of Proposition \ref{CS_ext_ps}), and is thus convex on $(\lambda^-, \infty)$.
Let \label{lambdamoinsm}
\begin{align}\label{eq:m}
\m:=\laplexp^\prime(0+)=
g-\sigma^2+ 2r\E\left[ \ln\Theta \right]
\end{align}
and \al{put} $\hat{\tau}=\argmin_{(\lambda^-, 0)}\laplexp(\lambda)$ which is well-defined if $\lambda^-<0<\m$ because $\laplexp'$ is an 
increasing function. We also define
\begin{align}\label{eq:kapparq}
\kapparq:=\laplexp(\hat{\tau})+r-q
= \hat{\tau} (g-\sigma^2)+ \hat{\tau}^2 \sigma^2 +  r \left[2\int_0^1 \theta^{\hat{\tau}} \kappa(d\theta) -1\right]-q.
\end{align}
We have the following sufficient condition for the mean number of cells to go to infinity.

\begin{prop} \label{CS_ext_ps}
Assume that the dynamics of the number of parasites in a cell follows \eqref{X_sans_sauts}, with $p(x) = x$, 
 and $\sigma(x)^2 = \mathfrak{s}^2(x) x + \sigma^2x^2$ for any $x \in \overline{\R}_+$ with $\sigma \in \R_+$.
 
Suppose that
\begin{itemize}
\item $q(x)\leq q < r$ and $g(x)\leq g x$ for any $x \in \overline{\R}_+$ with $g\in \R_+$.
\item the function $\mathfrak{s}$ is H\"older continuous with index $1/2$ on compact sets and there exists a finite \al{positive} constant $\mathfrak{c}$ such that
for $x \geq 0$, $\mathfrak{s}(x)\sqrt{x} \leq \mathfrak{c} \vee x^\mathfrak{c}$
\item for $x \geq 0$ the random variable $\Theta(x)$ is stochastically dominated by a random variable $\Theta$ satisfying \eqref{ass_Theta}.
\end{itemize}
Then, if $\m\leq 0$ or ($\m>0$ and $\kapparq>0$), for any $x> 0$
$$\lim_{t \to \infty} \E_{\delta_{x}}[\mathfrak{C}_t] = \infty. $$
\end{prop}

\al{\begin{rem}
In Proposition \ref{CS_ext_ps}, the Brownian coefficient of the dynamics of the parasites is $\sqrt{2\mathfrak{s}^2(x)x + 2\sigma^2x^2}$. This part of the dynamic can be decomposed into two different type of fluctuations: random fluctuations in the parasites growth (corresponding to the part $2\mathfrak{s}^2(x)x$) and the modeling of a random environment for the parasites (corresponding to $2\sigma^2x^2$).
\end{rem}}
Hence, for a large class of models, the cell population survives the infection with positive probability if the strategy for repartition of the parasites at division is well-chosen. The sign of $\m$ indicates if the number of parasites stays finite with a positive probability in a typical cell line. If it is the case ($\m\leq 0$), then the expected number of cells alive goes to infinity as time goes to infinity because the cell population grows exponentially at a rate larger than $r-q>0$. If $\m>0$, then the probability that the number of parasites is infinite in a typical cell line goes to $1$ as time goes to infinity. In that case, the speed of convergence of this probability has to be compared with the growth of the population. And if the growth of the population is strong enough ($\kapparq>0$), the expected number of cells that are alive still goes to infinity as time goes to infinity.

Focusing on the role of the partitioning kernel, Proposition \ref{CS_ext_ps} shows that if the cell population manages to adapt its partitioning strategy to make it more asymmetric, it can save the cell population (in the sense of making the mean number of cells alive tend to infinity for large time). Indeed, for any choice of the triplet $(g,\sigma,r)$, we can find a kernel $\kappa$ satisfying \eqref{ass_Theta} such that
$$ \E\left[ \ln\Theta \right]\leq \frac{\sigma^2-g}{2r},$$ 
and thus $\m\leq 0$.
\cha{It highlights that whatever the triplet $(g,\sigma,r)$, survival of the cell population with positive probability may be guaranteed by a kernel $\kappa$ which is sufficiently asymmetric.}

If we specify a bit more the dynamics of the number of parasites, we can give more precise results on the asymptotic behaviour of $\mathfrak{C}_t$. To that end, \al{instead of \eqref{X_sans_sauts}, we consider a simplified version of the SDE}:
\begin{equation} \label{X_sans_sauts2} \mathfrak{X}_t =x + g\int_0^t \mathfrak{X}_sds+ \int_0^t\sqrt{2\sigma^2\mathfrak{X}_s^2}dB_s
+\int_0^t\int_0^{\mathfrak{X}_{s^-}}\int_{\mathbb{R}_+}zR(ds,dx,dz),
\end{equation}
where $g \geq 0,\sigma \geq 0 $, $x \geq 0$, $B$ is a standard Brownian motion and the Poisson measure $R$ has been defined in \eqref{X_sans_sauts}.
In this case, we are able to obtain an equivalent of the mean number of cells alive at a large time $t$. It emphasizes how crucial is the choice of repartition of the parasites at division between daughter cells, \cha{as this later may be directly translated into the sign of $\m$, which discriminates between the different possible long term behaviours of the cell population.}

\begin{prop}\label{CNS_ext_ps}
Assume that the dynamics of the number of parasites in a cell follows \eqref{X_sans_sauts2}, that $\Theta(x) \overset{\mathcal{L}}{=}\Theta$ and satisfies \eqref{ass_Theta}, and that $q(x)\equiv q\geq 0$ with $q \neq r$.
\begin{enumerate}[label=\roman*), ref = {\it \roman*)}]
\item \label{it:1-22} If $\m<0$, then
for every $x>0$ there exists $0<c_1(x)<1$ such that
\begin{equation*}
\underset{t\rightarrow\infty}{\lim}e^{(q-r)t}\E_{\delta_{x}}\left[\mathfrak{C}_t\right]=c_1(x). 
\end{equation*}
\item \label{it:2-22} If $\m=0$ and $\lambda^-<0$,
then
for every $x>0$ there exists $c_2(x)>0$ such that
\begin{equation*}
\underset{t\rightarrow\infty}{\lim}\sqrt{t}e^{(q-r)t}\E_{\delta_{x}}\left[\mathfrak{C}_t\right]=c_2(x).  
\end{equation*}
\item \label{it:3-22} If $\m>0$, then
for every $x>0$ there exists $c_3(x)>0$ such that
\begin{equation*}
\underset{t\rightarrow\infty}{\lim}t^{\frac{3}{2}} e^{-\kapparq t} \E_{\delta_{x}}\left[\mathfrak{C}_t\right]=c_3(x). 
\end{equation*}
\end{enumerate}
\end{prop}

Note that the dependency on $\mathfrak{b}$ (parameter of the law of positive jumps for the parasites) is hidden in the limiting functions $c_1,c_2,c_3$.  We refer the reader to the proof for details.

In absence of parasites, if $r>q$, the cell population evolves as a supercritical Galton-Watson process and survives with probability $1-q/r$ \cite{athreya1972branching}. In the presence of parasites, the condition $r>q$ does not ensure that the cell population survives \al{with positive probability}, as it goes extinct almost surely if $\mathfrak m>0$ and $\kapparq\leq 0$. More generally, from the previous proposition, we deduce the following corollary on the asymptotic behaviour of $\mathfrak{C}_t$.

\begin{cor}\label{Cor_CNS_ext_ps}
Under the assumptions of Proposition \ref{CNS_ext_ps}, for any $x> 0$,
\begin{enumerate}[label=\roman*), ref = {\it \roman*)}]
\item \label{it:1-23} If $q\geq r$ or if $(\m>0 \text{ and }\kapparq \leq 0), $ then
$\lim_{t \to \infty} \E_{\delta_{x}}[\mathfrak{C}_t]=0$.
\item \label{it:2-23} If $(\m\leq 0\text{ and }r>q)$ or if $(\m>0 \text{ and }\kapparq >0),$
then $ \lim_{t \to \infty} \E_{\delta_{x}}[\mathfrak{C}_t]=\infty$.
\end{enumerate}
\end{cor}

Notice that the case $r=q$ is not taken into account in Proposition \ref{CNS_ext_ps} for the simplicity of its statement as it corresponds to a critical birth and death process for the cell population dynamics. However, we know that in this case the number of cells (with a finite or infinite number of parasites) reaches $0$ in finite time, hence the number of cells alive also reaches $0$ in finite time.\\

We now study the asymptotic qualitative behaviour ($0$ or $\infty)$ of $\E_{\delta_{x}}[\mathfrak{C}_t]$ as a function of $g$, $r$, $q$ and $\sigma$.
The dependence on $\kappa$ will be the subject of Section \ref{sec:role-partitioning} and will thus not be indicated here.
For the sake of readability and only for Lemma \ref{lem_dep}, we denote by 
$$ \E_{\delta_{x}}^{(g,\sigma,r,q)}[\mathfrak{C}_t] $$
the mean number of cells alive at time $t$ when there is initially one cell with a number $x$ of parasites, and that the dynamics of the infected cell population follows the assumptions of Proposition \ref{CNS_ext_ps} with parameters $(g,\sigma,r,q)$. We also introduce its limit when $t$ goes to infinity via
$$ \mathfrak{A}_x(g,\sigma,r,q):=\lim_{t \to \infty}\E_{\delta_{x}}^{(g,\sigma,r,q)}[\mathfrak{C}_t] $$
Then, we prove that for each parameter of the model, fixing all the other parameters, there exists a limiting value corresponding to a change of asymptotic behaviour of $\mathfrak{A}_x(g,\sigma,r,q)$.
\begin{lemma} \label{lem_dep}
Under the assumptions of Proposition \ref{CNS_ext_ps}, for any $x> 0$,
\begin{enumerate}[label=\roman*), ref = {\it \roman*)}]
\item There exists $q_{\lim}(g,\sigma,r) \in \R_+$ such that
\begin{align*}
q \geq q_{\lim}(g,\sigma,r) \Rightarrow \mathfrak{A}_x(g,\sigma,r,q) &= 0 \text{ and }
q < q_{\lim}(g,\sigma,r) \Rightarrow \mathfrak{A}_x(g,\sigma,r,q) = \infty.
\end{align*}
\item There exists $r_{\lim}(g,\sigma,r,q) \in \R_+$ such that
\begin{align*}
r \leq r_{\lim}(g,\sigma,q) \Rightarrow \mathfrak{A}_x(g,\sigma,r,q) &= 0 \text{ and }
r > r_{\lim}(g,\sigma,q) \Rightarrow \mathfrak{A}_x(g,\sigma,r,q) = \infty.
\end{align*}
\item There exists $g_{\lim}(\sigma,r,q) \in \R_+$ such that
\begin{align*}
g \geq g_{\lim}(\sigma,r,q) \Rightarrow \mathfrak{A}_x(g,\sigma,r,q) = 0  \text{ and }
g < g_{\lim}(\sigma,r,q) \Rightarrow \mathfrak{A}_x(g,\sigma,r,q) = \infty.
\end{align*}
\item There exists $\sigma_{\lim}(g,r,q) \in \bar{\R}_+$ such that
\begin{align*}
\sigma < \sigma_{\lim}(g,r,q) \Rightarrow \mathfrak{A}_x(g,\sigma,r,q) &= 0 \text{ and }
\sigma > \sigma_{\lim}(g,r,q) \Rightarrow \mathfrak{A}_x(g,\sigma,r,q) = \infty.
\end{align*}
Moreover,
$
\mathfrak{A}_x(g,\sigma_{\lim}(g,r,q),r,q) =\left\lbrace
\begin{array}{ll}
\infty, & \text{if }\ \sigma_{\lim}(g,r,q)=0,\\
0, & \text{otherwise.}
\end{array} \right. 
$
\end{enumerate}
\end{lemma}

The effect of $\kappa$, which describes the sharing of the parasites at division, is less intuitive and explicit computations are not always feasible. We nevertheless are able to study some particular cases and to compare some classes of partitioning kernels $\kappa$.

\section{Mean number of cells alive: Role of the partitioning kernel}\label{sec:role-partitioning}

\cha{In this section we further investigate the number of cells alive in large time, focusing on the effect of parasite sharing between daughter cells. We consider stronger assumptions than in the previous section, in order to obtain explicit expressions. \al{We assume that $\kappa(x,\cdot)\equiv\kappa(\cdot)$, {\it i.e.} that the partitioning is independent of the number of parasites.} The law of parasite sharing is then given by the random variable $\Theta$ whose law satisfies \eqref{ass_Theta}, and the number of parasites in a cell follows the SDE \eqref{X_sans_sauts2} with $\sigma=0$, that is
$$ \mathfrak{X}_t =x + g\int_0^t \mathfrak{X}_sds+ \int_0^t\int_0^{\mathfrak{X}_{s^-}}\int_{\mathbb{R}_+}zR(ds,dx,dz). $$}

\subsection{Deterministic vs random partitioning: a numerical study}
Focusing on two families of measures for the partitioning of the parasites at division, we first explore via numerical computations the \al{performance} of the different strategies in terms of survival of the cell population. Note that those families have also been considered in \cite{genthon2022} to study the effect of the partitioning kernels on the cell size distributions.

The first family (also considered in \cite{BPS}) corresponds to deterministic partitioning, with associated measures on $[0,1]$ denoted by $\kappa_z$, and given for all $z\in[0, 1/2]$ by
\begin{align*}
\kappa_z(d\theta) = \frac{1}{2}(\delta_z(d\theta) + \delta_{1-z}(d\theta)).
\end{align*}
The bigger $z$ is, the more the partitioning is asymmetric: $z=1/2$ corresponds to symmetric partitioning and $z=0$ to the case of all parasites going to one daughter cell.  For this family,
$ \m_{z}=g +   r\ln( z(1-z) ),$
where $\m$ is defined in \eqref{eq:m} and the subscript $z$ refers to the partitioning kernel $\kappa_{z}$.

The second family of measures that we consider corresponds to random partitioning. The associated partitioning measures on $[0,1]$, denoted by $\kappa_{\alpha}$, are given for all $\alpha>-1$ by
\begin{align}\label{eq:random-kernel}
\kappa_\alpha(d\theta) = c_\alpha\theta^\alpha(1-\theta)^\alpha d\theta,
\end{align}
where $c_\alpha = \Gamma(2\alpha + 2)/\Gamma(\alpha + 1)^2$, and $\Gamma$ is the Gamma function. In Figure \ref{fig:kernels}, the shapes of the kernels for various values of $\alpha$ are represented. The bigger $\alpha$ is, the more the partitioning is asymmetric. Note that $\alpha =0$ corresponds to the uniform sharing $\kappa(d\theta)=d\theta$, and will be studied in details below. 
\begin{figure}
\includegraphics[scale=0.25]{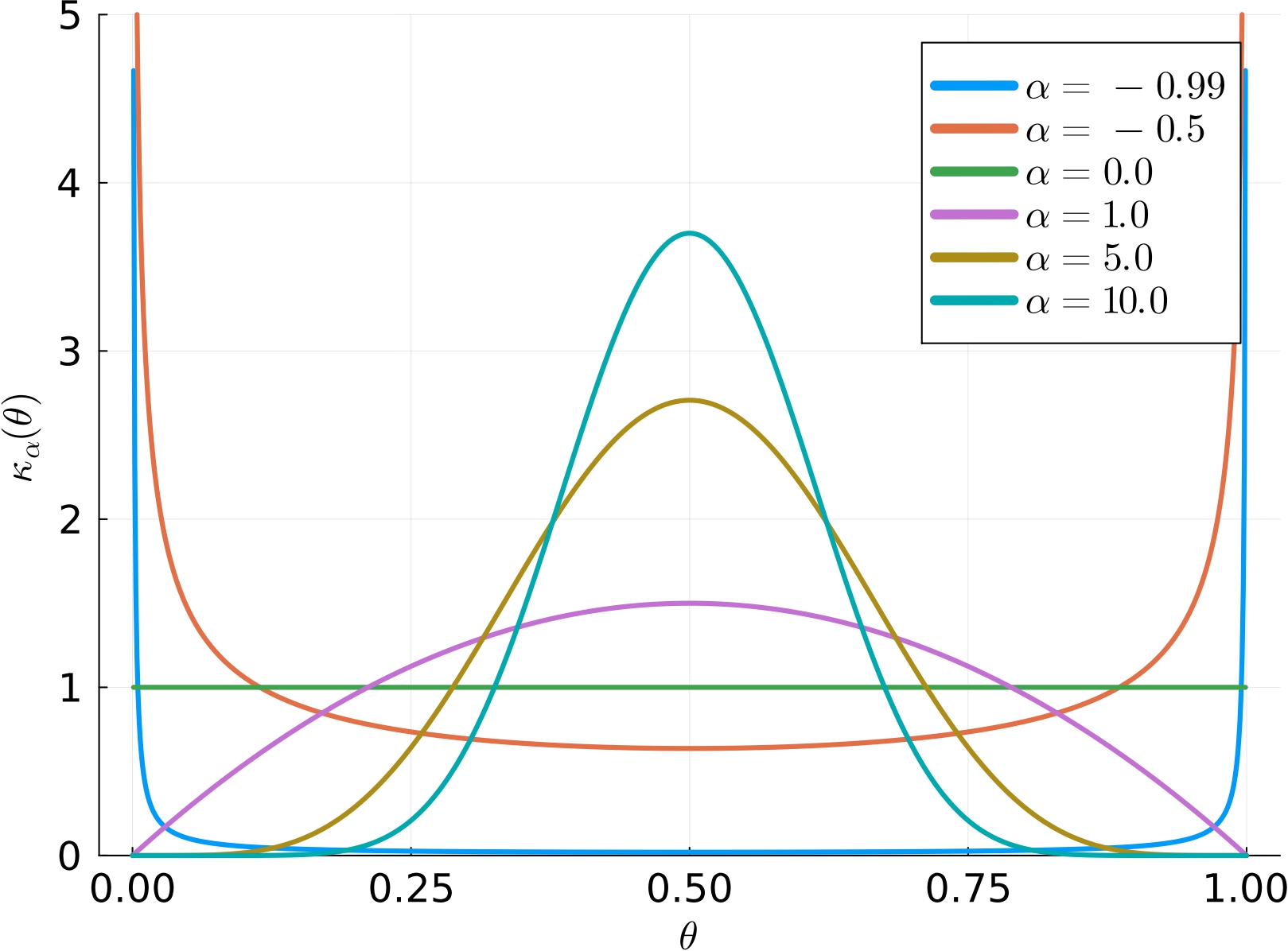}
\caption{The density functions of the partitioning kernels $\kappa_\alpha$ for various values of $\alpha$}
\label{fig:kernels}
\end{figure}
For any $\alpha>-1$, $\m_\alpha =g + 2r(\psi_0(\alpha) - \psi_0(2\alpha+1)),$
where $\psi_0$ is the digamma function and the subscript $\alpha$ refers to the partitioning kernel $\kappa_\alpha$.

We will now compare the two partitioning strategies, deterministic or random, represented by the two families of partitioning kernels $(\kappa_z,\ z\in~(0,1/2])$ and $(\kappa_\alpha,\ \alpha>-1)$. We denote by $\Theta_z$ and $\Theta_\alpha$ the random variables with distribution $\kappa_z$ and $\kappa_\alpha$ respectively.

First, we define $\mintheta_z :=\E\left[\min(\Theta_z,1-\Theta_z)\right]$ and $\mintheta_\alpha :=\E\left[\min(\Theta_\alpha,1-\Theta_\alpha)\right]$. This quantity corresponds to the expectation of the minimal fraction of parasites inherited by one of the daughter cells.
To each random partitioning $\kappa_\alpha$, we associate the deterministic partitioning kernel $\kappa_z$ such that $\mintheta_z=\mintheta_\alpha$.
Notice that the choice of $z$, for a given $\mintheta_\alpha$ is unique, as $\mintheta_z=z$. As a consequence, we
denote by $(\alpha, z_\alpha)=(\alpha, \mintheta_\alpha)$ such couples of parameters. Simple computations give
\begin{align*}
\mintheta_\alpha = 2c_\alpha\mathrm{B}(1/2; \alpha + 2,\alpha + 1), 
\end{align*}
where $\mathrm{B}(x;a, b):= \int_0^{x} \theta^{a-1}(1-\theta)^{b-1} d\theta $ is the incomplete Beta function ($x\in [0,1], a,b>-1$). Then,
\begin{align*}
z_\alpha = 2c_{\alpha} \mathrm{B}(1/2; \alpha + 2,\alpha + 1).
\end{align*}
Using this correspondence between the two families of kernels, we compare the two partitioning strategies (random or deterministic) for various values of $\mintheta$. In Figure \ref{fig:q0}, we show the correspondence between the long time behaviour of the mean cell population size and the values of $(g/r,\alpha)$ for $\kappa_{\alpha}$, and the values of $(g/r,z_\alpha)$ for $\kappa_{z_\alpha}$. We fixed $q=0$, but we get similar behaviours for all values of $q<r$.  
Interestingly, the two families of partitioning kernels exhibit the same qualitative behaviour
in terms of proliferation of the infection.
We observe that strategies with larger variances (small values of $z$ and $\alpha$) are more efficient in terms of survival of the cell population. For a given strength of proliferation of the infection $g$ relative to the growth of the population $r$, the fate of the cell population depends on the value of $\alpha$ or $z$: parameters
\begin{itemize}
\item in the green area lead to survival of the cell population, for both partitioning kernels,
\item in the red area lead to extinction of the cell population for both partitioning kernels,
\item in the orange area lead to extinction of the cell population for deterministic partitioning strategies, and survival of the cell population for random partitioning strategies.
\end{itemize}
Therefore, for any infection level and a given $\mintheta$, the random partitioning strategy is always better in terms of asymptotic mean number of cells alive in the population than deterministic partitioning ($\kapparq_{z_{\alpha}}<\kapparq_{\alpha}$). Moreover, there exist parameters $g,r$ ({\it e.g.} $\log(g/r)=4$ in Figure~\ref{fig:q0}) such that a population with deterministic partitioning gets extinct (in the sense of asymptotic mean number of cells alive) whatever the value of $z$, whereas a population with random partitioning can survive, if the division is sufficiently asymmetric.   
\begin{figure}
\includegraphics[scale=0.6]{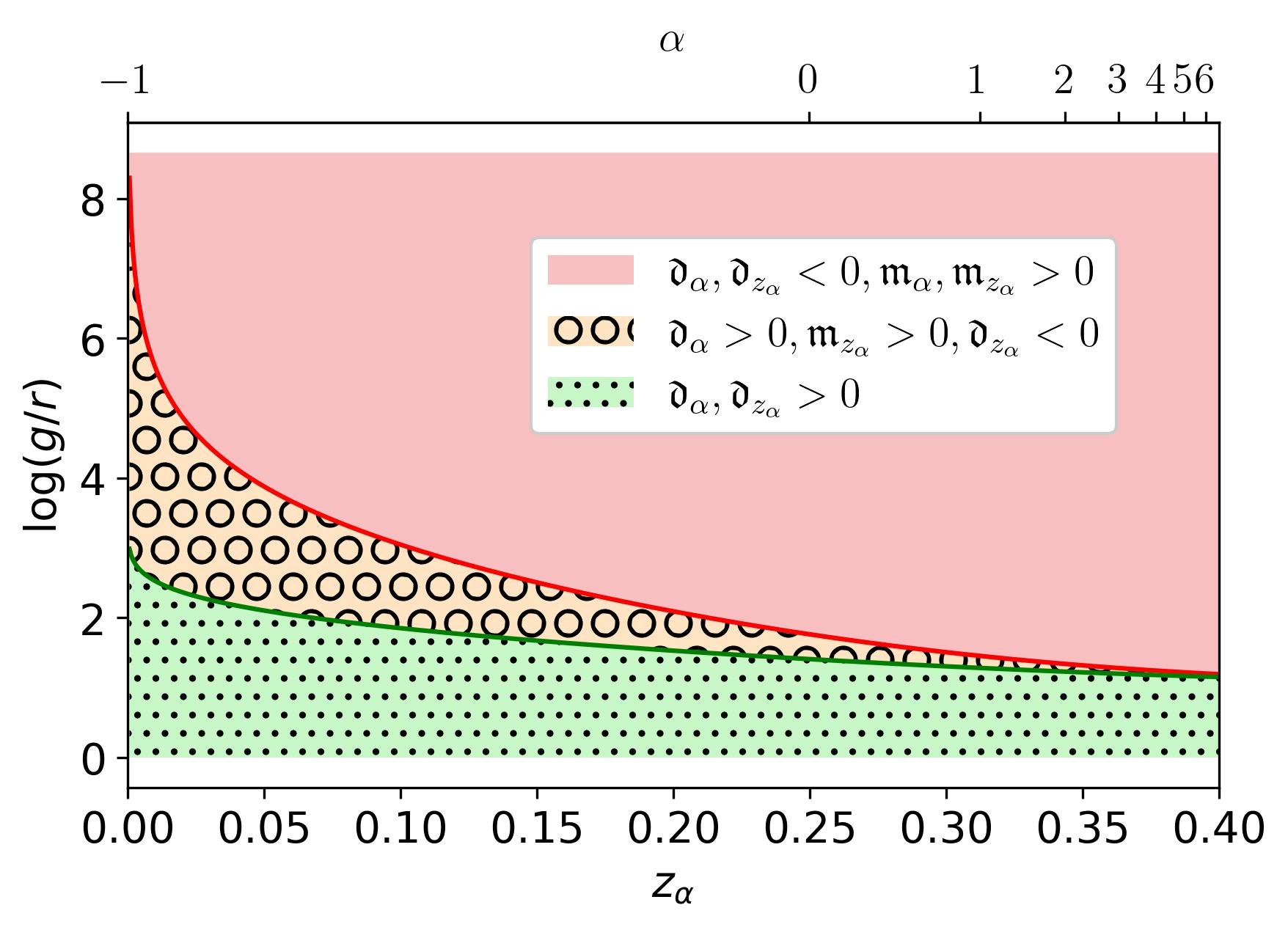}
\caption{Classification of the couples of parameters $(g/r, z_\alpha)$ and  $(g/r, \alpha)$ in the case of deterministic or random partitioning with $q=0$. \al{Parameters in the red empty and green dotted areas lead to survival or extinction of the cell population, respectively, for both deterministic and random partitioning kernels; parameters in the orange area with circles lead to extinction of the cell population for deterministic partitioning strategies, and to survival of the cell population for random strategies. For each $z_\alpha$ (the mean relative size of the smaller fragment), the corresponding value of $\alpha$ is given on the top horizontal axis}.}

\label{fig:q0}
\end{figure}
Note that to simplify the figure, we did no plot the curves $\m_{z_\alpha}=0$ and $m_{\alpha}=0$, but they behave similarly to the curves $\kapparq_{z_\alpha}=0$ and $\kapparq_{\alpha}=0$.

Finally, for high levels of the proliferation, neither the random nor the deterministic strategies considered here can overcome the infection, or with an extreme asymmetric distribution ($\mintheta \approx 0$). In Proposition~\ref{prop_bigg} below, we prove that for any value of $\vartheta$, there exists a partitioning strategy ensuring the survival of the population.

\subsection{Analytic comparison of partitioning strategies} \label{sec_analytic}
The most simple examples of partitioning strategies are the uniform law and the symmetric sharing, belonging respectively to the family of random and deterministic partitioning studied above. For those laws, we can explicit the bounds of Corollary \ref{Cor_CNS_ext_ps}.

\begin{cor}\label{Cor_CNS_ext_ps_unif}
Assume that the number of parasites in a cell follows the SDE \eqref{X_sans_sauts2} with $\sigma=0$, that $ r>q(x)\equiv q\geq 0$. Recall the definition of $g_{\lim}(\sigma,r,q)$ in Lemma \ref{lem_dep}.
\begin{itemize}
\item[-]If $\kappa(d\theta)=d\theta$, $$ g_{\lim}(0,r,q) =  3r-q+2\sqrt{2r(r-q)}. $$
\item[-]If $\kappa(d\theta)=\delta_{1/2}(d\theta)$, $$ g_{\lim}(0,r,q) =  r x_0(q/r)\ln 2 $$
where $x_0(q/r)>2$ is the unique value such that
\begin{align}
\label{eq:x0rq}
x_0(q/r) =  \left(1+\frac{q}{r}\right)(1+\ln 2- \ln \left( x_0(q/r)  \right))^{-1}.
\end{align}
\end{itemize}
\end{cor}

From this result, one can prove with a few more computations that the `uniform sharing' strategy is always better than the `equal sharing' strategy in terms of survival of the cell population. In fact, the symmetric sharing is the worst strategy, as stated in the next proposition. 
\begin{prop}\label{prop:inf}
Assume that the number of parasites in a cell follows the SDE \eqref{X_sans_sauts2} with $\sigma=0$, that $ r>q(x)\equiv q\geq 0$. For any partitioning kernel $\kappa$, if $g/r<x_0(q/r)\ln 2$, where $x_0(q/r)$ is defined in \eqref{eq:x0rq}, we have for all $x\geq 0$, 
\begin{align*}
\lim_{t\rightarrow+\infty}\E_{\delta_x}[\mathfrak{C}_t]=\infty. 
\end{align*}
\end{prop}
As $x_0(q/r)\ln 2$ is the limiting value corresponding to the case of an equal sharing, Proposition \ref{prop:inf} proves that any other sharing strategy is better than the symmetric partitioning.

More generally, we expect that a more unequal strategy is beneficial for the cell population: it amounts to `sacrificing' some lineages in order to 
save the other ones. We 
were not able to prove such a general statement, but we
will try to understand better the effect of unequal sharing in the next two propositions. Recall that $\mintheta=\E\left[\min(\Theta, 1-\Theta)\right]$. First, as explained above and in Figure \ref{fig:q0}, for a fixed value of $\mintheta$, random partitioning is always better than deterministic partitioning in terms of survival of the population. For a fixed value of $\mintheta$, is the deterministic partitioning the worst strategy in general? Second, does there exist, for any level of infection and for a fixed value of $\mintheta$, a partitioning distribution that leads to survival of the cell population?

 To approach these questions, we first consider finite points partitioning distributions for illustrative purposes. Let $n\geq 1$ be the number of possible modes $z_1, \ldots,z_n\in (0, 1/2]$, which are independent and identically distributed according to a uniform law on $(0,1/2]$. Then, the associated partitioning distribution has $2n$ modes, $z_1, \ldots,z_n, 1-z_1,\ldots,1-z_n$. Next, we define $p_i:=\P(\Theta=z_i)$ and $\mathbf{p}=(p_1,\ldots,p_n)$.
For those multimodal distributions, we have
\begin{align*}
\mintheta= \E\left[\min(\Theta, 1-\Theta)\right]= 2\sum_{i = 1}^nz_ip_i. 
\end{align*}
\begin{figure}
\includegraphics[scale=0.35]{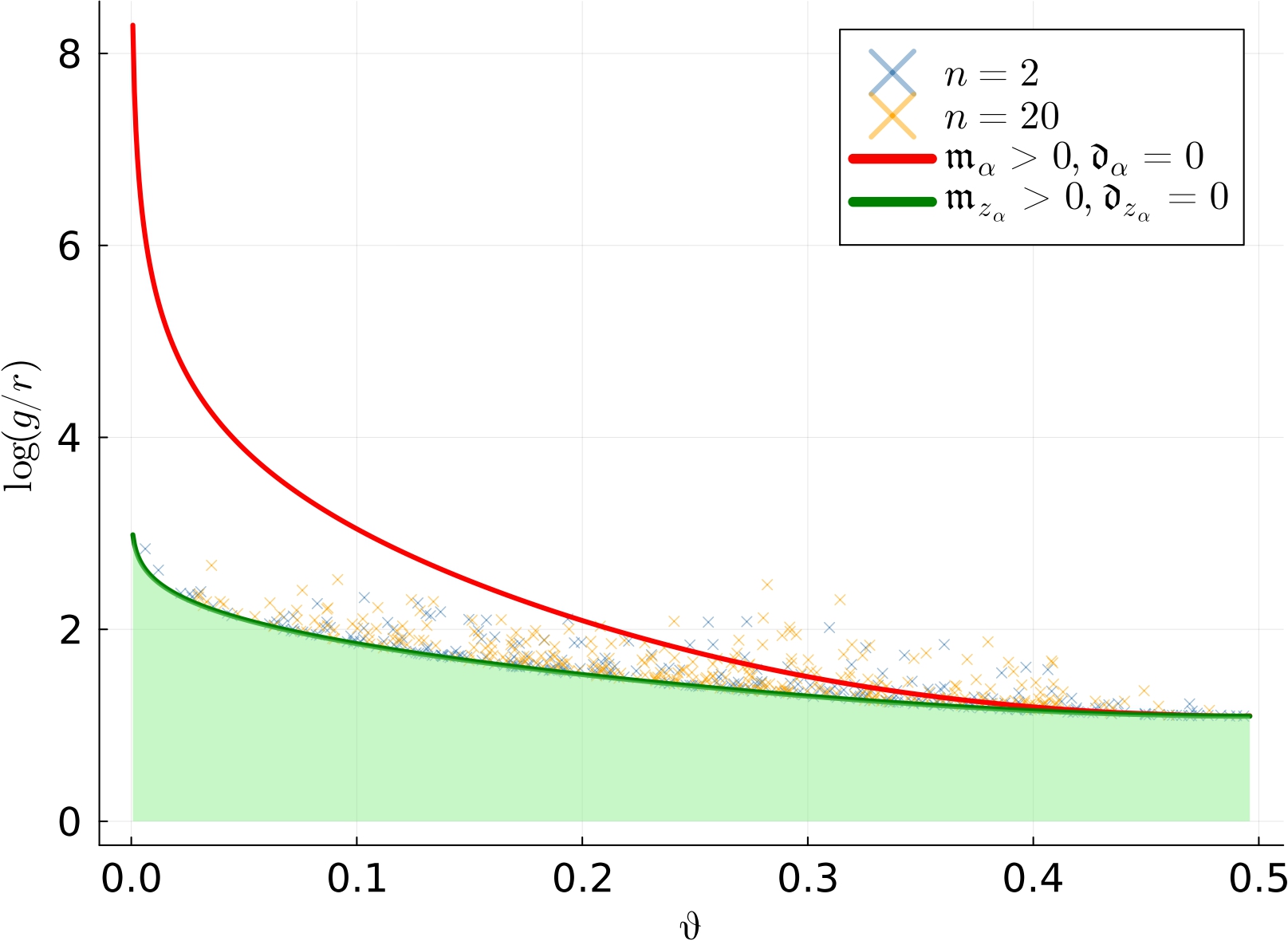}
\caption{Classification of the couples of parameters $(g/r, z_\alpha)$ and  $(g/r, \alpha)$ in the case of finite point partitioning distributions with $q=0$. Parameters in the green area lead to survival of the cell population for any finite point partitioning kernel. Each cross corresponds to the limit for a finite point distribution with a given value of $\vartheta$ above which the cell population go to extinction. The orange crosses (resp. blue) correspond to distributions with $20$ (resp. $2$) modes below $1/2$. The red curve corresponds to the limit above which a cell population with a random partitioning strategy (see \eqref{eq:random-kernel}) goes to extinction.}
\label{fig:multi}
\end{figure}
In Figure \ref{fig:multi}, we plot the logarithm of the limiting value $g/r$ at which $\kapparq =0$ for various multimodal distributions as a function of $\mintheta$.
 The green and red curves represent the limiting value for the kernels $\kappa_z$ and $\kappa_\alpha$ respectively, studied in the previous section. We observe that for a fixed value of $\mintheta$, the worst scenario seems to be the case of a deterministic partitioning $\kappa_z$. We will prove this result analytically for any symmetric distribution on $[0,1]$ . 

\begin{prop}\label{prop:multimodal}
Assume that the number of parasites in a cell follows the SDE \eqref{X_sans_sauts2} with $\sigma=0$, and that $ r>q(x)\equiv q\geq 0$. Let $\vartheta \in (0,1/2]$ and let $\kappa$ be a symmetric distribution on $[0,1]$  such that
$$\int_0^1\min (\theta , (1-\theta))\kappa(d\theta)= \vartheta.$$
Finally, let
$$ \kappa_{\vartheta}(d\theta)=1/2\left(\delta_{\vartheta}+\delta_{1-\vartheta}\right),$$
be the associated deterministic partitioning kernel.
Then, for any $x>0$,
$$ \lim_{t\rightarrow+\infty}\E_{\delta_x}^{(\kappa)}[\mathfrak{C}_t]\geq \lim_{t\rightarrow+\infty}\E_{\delta_x}^{(\kappa_{\vartheta})}[\mathfrak{C}_t], $$
where $\E^{(\kappa)}$ (resp. $\E^{(\kappa_\vartheta)}$) denotes the expectation for the population process with partitioning kernel $\kappa$ (resp. $\kappa_\vartheta$).
\end{prop}

 On the other hand, there is no upper bound: for any value of $\mintheta\in (0,1/2)$, and any value of $y\geq 0$, one can find a finite point measure (with $n=2$ for example) such that for all $g/r<y$, the mean number of cells alive goes to infinity when time goes to infinity. This can be achieved by taking very small values for $z_1$, which is the smallest atom of the partitioning distribution. 
This is formally stated in the following proposition.
\begin{prop} \label{prop_bigg}
Let $g,r,q \in \R_+$ with $q<r$ and $\mintheta\in (0,1/2)$. Then,
there exists a multimodal distribution 
$$\kappa_2(d\theta)=\sum_{i=1}^2\left(\delta_{z_i}(d\theta)+\delta_{1-z_i}(d\theta)\right)p_i,$$
with $2(p_1+p_2)=1$, and $(z_1,z_2) \in  (0,1/2)^2 $, such that if $\Theta \sim \kappa_2 $,
$$  \E\left[\min(\Theta, 1-\Theta)\right]=\mintheta \quad \text{and}\quad \lim_{t \to \infty} \E_{\delta_{x}}[\mathfrak{C}_t] = \infty \quad \text{for any} \quad x> 0. $$
\end{prop}

\section{Number of parasites in the cells} \label{section_beta}

We now consider that the dynamics of the parasites in a cell follows the SDE \eqref{X_sans_sauts} without the stable positive jumps, that is to say
\begin{align} \label{X_sans_sauts_stables} \mathfrak{X}_t =x + \int_0^t g(\mathfrak{X}_s)ds
+\int_0^t\sqrt{2\sigma^2 (\mathfrak{X}_s)}dB_s &+
\int_0^t\int_0^{p(\mathfrak{X}_{s^-})}\int_{\mathbb{R}_+}z\widetilde{Q}(ds,dx,dz).
\end{align}

In this case we can observe moderate infections, extinctions of the parasites in the cell population, but also cases where the number of parasites goes to infinity with 
an exponential growth in a positive fraction of the cells.\\

In order to state the next result, we need to introduce three assumptions. The first one is a technical assumption allowing to make couplings, that could probably be weakened.

\begin{ass} \label{ass_D}
The measure $\pi$ satisfies $ \int_{\R_+} z\pi(dz)<\infty.$
\end{ass}
Note that the weaker condition $ \int_{\R_+} \ln(1+z)\pi(dz)<\infty $, which is required in \cite{companion}, is therefore satisfied under Assumption \ref{ass_D}.
The second assumption provides a condition under which the number of parasites may reach $0$. It is almost a necessary and sufficient condition (see  \cite[Remark 3.2 and Theorem 3.3]{companion}).

\begin{enumerate}[label=\bf{(LN0)}]
\item\label{A2}  There exist $0<a<1,$ $\eta>0$ and $x_0> 0$ such that for all $x\leq x_0$
\begin{equation*}
\frac{g(x)}{x}-a\frac{\sigma^2(x)}{x^2}-2r\frac{1-\E\left[\Theta^{1-a}(x)\right]}{1-a} \leq -\ln(x^{-1}) \left(\ln\ln(x^{-1})\right)^{1+\eta}.
\end{equation*}
\end{enumerate}
The third assumption ensures that the process does not explode in finite time almost surely (see \cite[Theorem 4.1]{companion}). 
\begin{enumerate}[label=\bf{(SN$\infty$)}]
\item\label{SNinfinity}  There exist $0<a<1,$ and a nonnegative function $f$ on $\mathbb{R}_+$ such that
\begin{equation*}
\frac{g(x)}{x}-a\frac{\sigma^2(x)}{x^2}-2r\frac{1-\E\left[\Theta^{1-a}(x)\right]}{1-a}-p(x)I_a(x) =-f(x)+o(\ln x),\quad (x\rightarrow +\infty),
\end{equation*}
where
 \begin{align*}\label{eq:Ia}
%C_{a}=\frac{1}{2+\mathfrak{b}}\int_{\mathbb{R}_+}\frac{z}{(1+z)^{a}}\rho(dz),\quad 
%=\int_{\mathbb{R}_+}z\int_0^1(1+vz)^{-a}dv\rho(dz)
I_a(x) = \frac{a}{x^{2}}\int_{\mathbb{R}_+}z^2\left(\int_0^1\frac{(1-v)}{(1 + zvx^{-1})^{1+a}}dv\right)\pi(dz).
\end{align*}
\end{enumerate}

Note that the term $-2r\big(1-\E\left[\Theta(x)^{1-a}\right]\big)(1-a)^{-1} $ is not present in {\bf{(LN0)}} and {\bf{(SN$\infty$)}} in \cite{companion} as it is a constant when the law of $\Theta(x)$ does not depend on $x$. However, the extension of the proofs of \cite{companion} to this case is possible under the additional assumption \ref{ass_LB} (Lower Bound) on the partitioning kernel.
\begin{customass}{\bf{LB}}\label{ass_LB}
There exists a symmetric random variable $\Theta$ on $[0,1]$ and a constant $c>0$ such that $$\inf_{x\geq 0}\Theta(x)\geq c\Theta.$$
\end{customass}

Recall that the total number of cells is given by a continuous-time birth and death process with individual birth rate $r$ and individual death rate $q$. 
From classical results on branching processes (see for instance \cite{athreya1972branching}), we know that the cell population survives with probability
$0 \vee (1-q/r)$.
The long time behaviour for the number of parasites in the cells is described in the next proposition. We denote by $N_t$ the cardinality of $V_t$. 

\begin{prop} \label{prop:beta_ct:g_var}
Assume that the number of parasites in a cell follows the SDE \eqref{X_sans_sauts_stables}, that Assumption \ref{ass_D} holds, that $r>q \equiv q(x)\geq 0$.
\begin{enumerate}[label=\roman*), ref = {\it \roman*)}]
 \item \label{it:1-31} If $\sup_{x \geq 0} \E[\ln^2 \Theta(x)]<\infty $, and there exists $\eta>0$ such that for $x \geq 0$,  $$\frac{g(x)}{x}+ 2 r \E[\ln \Theta(x)]> \eta,$$
 and if the function $x \mapsto (\sigma^2(x)+p(x))/x$ is bounded and there exists $\eps_1>0$ such that $$\int_{\R_+}z \ln^{1+\eps_1}(1+z)\pi(dz)<\infty,$$
 then for $\eps>0$,
\begin{align*}
\liminf_{t\rightarrow \infty} \mathbb{E}\left[ \mathbf{1}_{\{ N_t\geq 1 \}} \frac{\#\lbrace u\in V_t:X_t^u >e^{(\eta/2r-\eps)t}\rbrace}
{N_t} \right] > 0.
\end{align*} 
\item \label{it:2-31} If Assumption \ref{ass_LB} and \ref{A2} hold, and if there exists $\eta>0$ such that for all $x \geq 0$, $$\frac{g(x)}{x}+ 2 r \E[\ln \Theta(x)]<- \eta,$$ then for $\eps>0$ 
$$ \lim_{t\rightarrow \infty}\mathbf{1}_{\{ N_t\geq 1 \}}\frac{\#\lbrace u\in V_t:X_t^u >\eps \rbrace}
{N_t} = 0 \quad \text{in probability}.
 $$
\item \label{it:3-31} If Assumption \ref{ass_LB}, \ref{A2} and \ref{SNinfinity} hold, and if there exist $\eta>0$ and $x_0 \geq 0$ such that for $x \geq x_0$,  
$$\frac{g(x)}{x} - \frac{\sigma^2(x)}{x^2}+ 2 r \E[\ln \Theta(x)]-p(x)\int_0^\infty\left(\frac{z}{x}-\ln\left(1+\frac{z}{x}\right)\right)\pi(dz)<- \eta, $$
then
\begin{equation*}
\lim_{t\rightarrow \infty}\mathbf{1}_{\{ N_t\geq 1 \}}\frac{\#\lbrace u\in V_t: X_t^u >0\rbrace}
{N_t} = 0 \quad a.s.
\end{equation*}
 \end{enumerate}
\end{prop}

Proposition \ref{prop:beta_ct:g_var} extends \cite[Theorem 4.2]{BT11} 
allowing for non constant drift for the number of parasites, a general class of diffusive functions, positive jumps, a parasites repartition kernel depending on the number of parasites carried by the mother, as well as the possibility for the cells to die at a constant rate.

Again, from Proposition \ref{prop:beta_ct:g_var} we see that in some sense an equal sharing is the worst strategy at the population level. Indeed, from the concavity of the functions $x \mapsto \ln x $ and $x \mapsto x^{1-a}$ for $0\leq a<1$, we can prove that if the proportion of highly infected cells is positive for large time (Proposition \ref{prop:beta_ct:g_var}\ref{it:1-31}) with a given partitioning strategy, then the equal sharing strategy would have led to the same result. Conversely, if the equal sharing strategy guarantees the healing of the cell population for large time, then it would have been the case for any partitioning strategy.
\begin{lemma}
Under the assumptions of Proposition \ref{prop:beta_ct:g_var}, we have the following:
\begin{enumerate}[label=\roman*), ref = {\it \roman*)}]
\item If there exists a real function $(x, \zeta) \mapsto \Theta(x, \zeta)$ such that the assumptions of Proposition \ref{prop:beta_ct:g_var}\ref{it:1-31} hold, then they also hold for the equal sharing (corresponding to $\Theta \equiv 1/2$).
\item If the assumptions of point \ref{it:2-31} (resp. \ref{it:3-31}) of Proposition \ref{prop:beta_ct:g_var} hold for the equal sharing ($\Theta \equiv 1/2$), then they also hold for any real function $(x, \zeta) \mapsto \Theta(x, \zeta)$ such that $\E[\Theta(x)] = 1/2$ and $|\E[\ln \Theta(x)]|<\infty$.
\end{enumerate}
\end{lemma}
  
The rest of the paper is dedicated to the proofs of the results presented in previous sections. As mentioned before, the proofs 
rely on the construction of an auxiliary process, which gives information on the dynamics of the number of parasites in a `typical' cell, 
that is to say a cell chosen uniformly at random among the cells alive. 
\section{Many-to-One formula}\label{sec:MTO}

Recall from \eqref{Ztdirac} that the population state $Z_t$ at time $t$ can be represented by a sum of Dirac masses. 
 We denote by $(M_t,t\geq 0)$ the first-moment semi-group associated with the population 
process $Z$ given for all measurable functions $f$ and $x,t\geq 0$ by
$$
M_tf(x)=\mathbb{E}_{\delta_x}\left[\sum_{u\in V_t} f(X_t^u)\right].
$$ 
The trait of a typical individual in the population is characterized by the so-called 
auxiliary process $Y$ (see \cite[Theorem 3.1]{marguet2016uniform} for detailed computations and proofs). In our case, for constant birth and death rate, $Y$ is a time-homogeneous Markov process and for all measurable bounded functions $F:\mathbb{D}([0,t],\mathbb{R}_+)\rightarrow\mathbb{R}$, we have:
\begin{equation}\label{eq:mto}
\E_{\delta_{x}}\left[\sum_{u\in V_{t}}F\left(X_{s}^{u},s\leq t\right)\right]=e^{(r-q)t}\E_{x}\left[F\left(Y_{s}^{(t)},s\leq t\right)\right].
\end{equation}
Here $(Y_{t}, t\geq 0)$ is a Markov process with associated infinitesimal generator 
$\mathcal{A}$ given for $f\in\mathcal{C}_b^2(\mathbb{R}_+)$ and $x\geq 0$ by: 
\begin{align*}
\mathcal{A}f(x)= &\mathcal{G}f(x) + 2r\int_{0}^1\left(f\left(\theta x\right)-f\left(x\right)\right)\kappa(x,d\theta),
\end{align*}
We refer the reader \cite[Section 4.2]{marguet2022spread} for details on the role of the death rate in the auxiliary process.

\section{Proofs}\label{sec:proofs}
\subsection{Proofs of Section \ref{sec_mean_numb_cells}}

\begin{proof}[Proof of Proposition \ref{CNS_ext_ps}]
Let us consider the auxiliary process introduced in Section \ref{sec:MTO} as the unique strong solution to the following SDE:
\begin{align} \label{eq:Y-proof3.2}
Y_t= x + g\int_0^tY_s ds &+ \int_0^t \sqrt{2\sigma^2Y_s^2}dB_s + \int_0^t \int_0^{Y_{s^-}}
\int_{\mathbb{R}_+}zR(ds,dx,dz) \nonumber \\
&+\int_0^t\int_0^1  (\theta-1)Y_{s^-}P(ds,d\theta),
\end{align}
where $P$ is a Poisson point measure on $\R_+\times[0,1]$ with intensity $2rds\otimes \kappa(d\theta)$.
We can thus apply \eqref{eq:mto} to the function 
$$F((X_s^u,s \leq t))= \mathbf{1}_{\{X_t^u <\infty\}},$$
and obtain
\begin{equation*}
\E_{\delta_{x}}\left[\mathfrak{C}_t\right]=e^{(r-q)t}\P_{x}\left(Y_{t}<\infty\right),
\end{equation*}
where we recall that $\mathfrak{C}_t$ is the number of cells alive at time $t$ (that is to say containing a finite number of parasites).
The study of the asymptotic behaviour of $\E\left[\mathfrak{C}_t\right]$ is thus reduced to the study of the asymptotics of the non-explosion probability of $Y$.
Following \cite{palau2016asymptotic}, the long time behaviour of $\P_{x}\left(Y_{t}<\infty\right)$ depends on the properties of the L\'evy process $L$ given by:
\begin{equation} \label{def_Lt} 
L_t:=(g-\sigma^2)t + \sqrt{2 \sigma^2}B_t + \int_0^t \int_0^1  \ln \theta P(ds,d\theta),
\end{equation}
where $B$ and \al{$P$} are the same as in \eqref{eq:Y-proof3.2}.
Its Laplace exponent $\laplexp$ is 
\begin{equation*}
\laplexp(\lambda):=\ln \E[e^{\lambda L_1}]
= \lambda (g-\sigma^2) + \lambda^2 \sigma^2 + 2 r \left[\int_0^1 \theta^\lambda \kappa(d\theta) -1\right],
\end{equation*}
for any $\lambda \in (\lambda^-, \infty)$. Recall that $\lambda^-$ and $\m$ have been defined on page \pageref{lambdamoinsm}. Then an application of \cite[Proposition 2.1]{palau2016asymptotic} gives
the three following regimes:
\begin{itemize}
\item[i)] If $\m<0$, then
for every $x>0$ there exists $0<c_1(x)<1$ such that
\begin{equation*}
\underset{t\rightarrow\infty}{\lim}\P_{x}(Y_t<\infty)=c_1(x). 
\end{equation*}
\item[ii)] If $\m=0$ and $\lambda^{-}<0$,
then
for every $x>0$ there exists $c_2(x)>0$ such that
\begin{equation*}
\underset{t\rightarrow\infty}{\lim}\sqrt{t}\P_{x}(Y_t<\infty)=c_2(x).  
\end{equation*}
\item[iii)] If $\m>0$, then
for every $x>0$ there exists $c_3(x)>0$ such that
\begin{equation*}
\underset{t\rightarrow\infty}{\lim}t^{\frac{3}{2}} e^{-\laplexp(\hat{\tau})}\P_{x}(Y_t<\infty)=c_3(x). 
\end{equation*}
\end{itemize}
It ends the proof.
\end{proof}
 
\begin{proof}[Proof of Proposition \ref{CS_ext_ps}]
First, we consider the case $g(x) = gx$ and $q(x)\equiv q$. 
The process $\mathfrak{X}$ solution to \eqref{X_sans_sauts} has the same law as the unique solution to the SDE
\begin{align*} \tilde{\mathfrak{X}}_t =x +& g\int_0^t \tilde{\mathfrak{X}}_sds
+\int_0^t\sqrt{2\sigma^2 \tilde{\mathfrak{X}}^2_s}dB_s+\int_0^t\sqrt{2\mathfrak{s}^2 (\tilde{\mathfrak{X}}_s)\tilde{\mathfrak{X}}_s}dW_s \\+&
\int_0^t\int_0^{\tilde{\mathfrak{X}}_{s^-}}\int_{\mathbb{R}_+}z\widetilde{Q}(ds,dx,dz)+
\int_0^t\int_0^{\tilde{\mathfrak{X}}_{s^-}}\int_{\mathbb{R}_+}zR(ds,dx,dz),
\end{align*}
where $W$ is a Brownian motion independent of $B$, $Q$ and $R$.
Notice that under the assumptions of Proposition \ref{CS_ext_ps}, $y \mapsto \mathfrak{s}(y)\sqrt{y}$ satisfies point $ii)$ of Assumption \ref{ass_A}.
As in the previous case, explicit computations are possible, and if we keep the notation $Y$ for the auxiliary process associated to 
$\tilde{\mathfrak{X}}$ for the sake of simplicity, we obtain that $Y$ is solution to:
 \begin{align} \label{X_sans_sauts4} Y_t =&x + g \int_0^t Y_sds
+\int_0^t\sqrt{2\sigma^2 Y^2_s}dB_s+\int_0^t\sqrt{2\mathfrak{s}^2 (Y_s)Y_s}dW_s+
\int_0^t\int_0^{Y_{s^-}}\int_{\mathbb{R}_+}z\widetilde{Q}(ds,dx,dz)\nonumber\\&+
\int_0^t\int_0^{Y_{s^-}}\int_{\mathbb{R}_+}zR(ds,dx,dz)+\int_0^t \int_0^1  (\Theta(Y_{s^-},\zeta)-1)Y_{s^-}N(ds,d\zeta)\\=&\cha{x + g \int_0^t Y_sds
+\int_0^t\sqrt{2\mathfrak{s}^2 (Y_s)Y_s}dW_s+
\int_0^t\int_0^{Y_{s^-}}\int_{\mathbb{R}_+}z\widetilde{Q}(ds,dx,dz)\nonumber}\\&\cha{+
\int_0^t\int_0^{Y_{s^-}}\int_{\mathbb{R}_+}zR(ds,dx,dz)+\int_0^t Y_{s^-}\left(\sqrt{2\sigma^2}dB_s+\int_0^1  (\Theta(Y_{s^-},\zeta)-1)N(ds,d\zeta)\right),} \nonumber
\end{align}
where $Y_0=x\geq 0$ and $N$ is \cha{a} PPM on $\R_+\times[0,1]$ with intensity $2rds\otimes d\zeta$, \cha{and independent of $B,W,R$ and $Q$}. 

Let us introduce the process $(\tilde{L}_t, t \geq 0)$ via
\begin{equation} \label{def_tildeLt} 
\tilde{L}_t:=(g-\sigma^2)t + \sqrt{2 \sigma^2}B_t + \int_0^t \int_0^1  \ln \Theta(Y_{s^-},\zeta) N(ds,d\zeta).
\end{equation}

 Then by an application of It\^o's formula with jumps similarly as in \cite{palau2018branching} 
we can show that for any $x,\lambda,0 \leq s \leq t$,
\begin{align}\label{eq:expZt}
 e^{-Y_se^{-\tilde{L}_s}v_t(s,\lambda, \tilde{L})} = \int_0^s e^{-Y_ue^{-\tilde{L}_u}v_t(u,\lambda, \tilde{L})}
e^{-2\tilde{L}_u}v^2_t(u,\lambda, \tilde{L})\mathfrak{s}^2(Y_u)Y_udu+ \mathfrak{M}_s,
\end{align}
where $(\mathfrak{M}_s, 0 \leq s \leq t)$ is a local martingale conditionally on $(\tilde{L}_s,0\leq s \leq t)$ and
$v_t(.,\lambda,\tilde{L})$ is the unique solution to
$$ \partial_s v_t(s,\lambda,\tilde{L})=e^{\tilde{L}_s}\psi_0\left(e^{-\tilde{L}_s}v_t(s,\lambda,\tilde{L})\right), \quad v_t(t,\lambda,\tilde{L})=\lambda, $$
where
$$\psi_0(\lambda) = c_\mathfrak{b} \lambda^{1+\mathfrak{b}} +\int_0^\infty \left( e^{-\lambda z}-1+\lambda z \right)\pi(dz).$$
With our assumptions on the function $\mathfrak{s}$, the process 
$$ \left( \exp\left(-Y_ue^{-\tilde{L}_u}v_t(u,\lambda, \tilde{L})\right)
e^{-2\tilde{L}_u}v^2_t(u,\lambda,\tilde{L})\mathfrak{s}^2(Y_u)Y_u,0\leq u \leq t\right) $$
is bounded by a finite quantity depending only on $(\tilde{L}_u,0\leq u \leq t)$ (using that $x\mapsto e^{-x}$ and $x\mapsto e^{-x}x^\mathfrak{c}$ are bounded on $\R_+$, where $\mathfrak{c}$ is defined in the assumptions of Proposition \ref{CS_ext_ps}).
Hence $(\mathfrak{M}_s, 0 \leq s \leq t)$ is a true martingale conditionally on $(\tilde{L}_s,0\leq s \leq t)$, and from \eqref{eq:expZt} we get
\begin{equation} \label{exp_Lap_CSBPRE}\E_x \left[ e^{-\lambda Y_te^{-\tilde{L}_t}} \right]= \E_x \left[ e^{- Y_te^{-\tilde{L}_t}v_t(t,\lambda,\tilde{L})} \right]\geq
\E \left[ e^{-xv_t(0,\lambda,\tilde{L})} \right]. \end{equation}
Using that $\psi_0(\lambda)>c_\mathfrak{b}\lambda^{1+\mathfrak{b}}$, we obtain
\begin{align*}
 \partial_s v_t(s,\lambda,\tilde{L})\geq c_\mathfrak{b} e^{\tilde{L}_s}\left(e^{-\tilde{L}_s}v_t(s,\lambda,\tilde{L})\right)^{1+\mathfrak{b}}, \quad v_t(t,\lambda,\tilde{L})=\lambda,
\end{align*}
which entails
$$ v_t(0,\lambda,\tilde{L}) \leq \left( \lambda^{-\mathfrak{b}}+\mathfrak{b} c_\mathfrak{b} \int_0^t e^{-\mathfrak{b} \tilde{L}_s}ds \right)^{-1/\mathfrak{b}}. $$
Combining this latter with \eqref{exp_Lap_CSBPRE}, we obtain
$$  \E_x \left[ e^{-\lambda Y_te^{-\tilde{L}_t}} \right]\geq  \E \left[ e^{-x\left( \lambda^{-\mathfrak{b}}+\mathfrak{b} c_\mathfrak{b} \int_0^t e^{-\mathfrak{b} \tilde{L}_s}ds 
\right)^{-1/\mathfrak{b}}} \right], $$
and letting $\lambda$ tend to $0$, we finally get:
$$  \P_x \left( Y_t<\infty \right)\geq  \E \left[ e^{-x\left( \mathfrak{b} c_\mathfrak{b} \int_0^t e^{-\mathfrak{b} \tilde{L}_s}ds \right)^{-1/\mathfrak{b}}} \right]. $$
From the assumptions of Proposition \ref{CS_ext_ps} we see that we can couple the processes $L$ and $\tilde{L}$, defined in \eqref{def_Lt} and \eqref{def_tildeLt}, respectively, in such a way that
$$ \tilde{L}_t \leq L_t \quad \text{a.s. for all } t \geq 0. $$
We thus deduce that
$$  \P_x \left( Y_t<\infty \right)\geq  \E \left[ e^{-x\left( \mathfrak{b} c_\mathfrak{b} \int_0^t e^{-\mathfrak{b} L_s}ds \right)^{-1/\mathfrak{b}}} \right]. $$
As stated in \cite{palau2016asymptotic}, the right-hand side of the last inequality is equal to the probability of non-explosion before time $t$ \cha{of} 
a self-similar continuous state branching process in a L\'evy random environment.
Therefore, by  \cite[Proposition 2.1]{palau2016asymptotic}, we get
\begin{equation}\label{mathfrakcx}
\underset{t\rightarrow +\infty}{\liminf}v(\m,t)\P_x \left( Y_t<\infty \right) =: \mathfrak{a}(x)>0,
\end{equation}
where
\begin{align*}\left\lbrace
\begin{array}{ll}
v(\m,t) = 1,&\quad \text{for }\m<0,\\
v(0,t) = \sqrt{t},&\\
v(\m,t) = t^{3/2}e^{t\laplexp(\hat{\tau})},&\quad \text{for }\m>0.\\
\end{array}
\right.
\end{align*}
Next, we consider the auxiliary process $\tilde{Y}$ in the case where the number of parasites is described by \eqref{X_sans_sauts}, 
with $p(x)=x$, $\sigma^2(x) = \mathfrak{s}^2(x) x +\sigma^2 x^2$ and $g(x)\leq gx$. In this case $\tilde{Y}$ has the same law as a process satisfying 
\eqref{X_sans_sauts4} replacing $g\int_0^tY_sds$ by $\int_0^tg(Y_s)ds\leq g\int_0^tY_sds$. Hence if we choose this version of $\tilde{Y}$, $\tilde{Y}_t\leq Y_t$ for all $t\geq 0$ using that both SDEs have a unique strong solution and 
that $\tilde{Y_0}=Y_0$. Therefore, 
$$
\mathbb{P}_x(\tilde{Y}_t<\infty)\geq \mathbb{P}_x(Y_t<\infty).
$$
As a consequence, from the Many-to-One formula \eqref{eq:mto} and the assumption that $q(\cdot) \equiv q$, we obtain for any $x >0$ and $t$ large enough:
\begin{align*}
 \E_{\delta_{x}}[\mathfrak{C}_t] = e^{(r-q)t}\mathbb{P}_x(\tilde{Y}_t<\infty)
& \geq  e^{(r-q)t} \mathbb{P}_x(Y_t<\infty)\\
& = e^{(r-q)t} v^{-1}(\m,t) \left( v(\m,t)\P_x \left( Y_t<\infty \right)\right)\\
& \geq e^{(r-q)t} v^{-1}(\m,t) \mathfrak{a}(x)/2,
\end{align*}
where we recall that $\mathfrak{a}(x)$ has been defined in \eqref{mathfrakcx}.
Adding that either $(\m>0\text{ and }\kapparq>0)$ or $\m\leq 0$ holds under the assumptions of Proposition \ref{CS_ext_ps}, we obtain that 
$$ \lim_{t \to \infty} \E_{\delta_{x}}[\mathfrak{C}_t] = \infty. $$
Now let us come back to the general case where for any $x \geq 0$, $
q(x) \leq q$ for some $q\geq 0$. Then for any $x > 0$ we can couple the process $X$ with 
a process $X^{(q)}$ with death rate $q$ and number of cells alive at time $t$ given by $\mathfrak{C}_t^{(q)}$, and such that
$$ \E_{\delta_{x}}[\mathfrak{C}_t] \geq \E_{\delta_{x}}[\mathfrak{C}^{(q)}_t]. $$
Such a coupling may be obtained for instance by first realizing $X$ and then obtaining $X^{(q)}$ by killing additional cells at rate $q-q(x)$ for a cell 
containing a number $x$ of parasites. It ends the proof.
\end{proof}

\begin{proof}[Proof of Lemma \ref{lem_dep}]
\textbf{$i)$}
Let $\mathfrak c :=\hat{\tau} (g-\sigma^2)+ \hat{\tau}^2 \sigma^2 +  r \left[2\E[ \Theta^{\hat{\tau}}] -1\right]$, so that $\kapparq = \mathfrak c -q$. The value of $\m$ does not depend on $q$. Hence, we distinguish three cases:
\begin{enumerate}
\item If $\m>0$ and $\mathfrak c\leq 0$, then for all $q\geq 0,$ $\kapparq\leq 0$ and $\mathfrak{A}_x(g,\sigma,r,q)=0$: we thus choose $q_{\lim}(g,\sigma,r)=0$.
\item If $\m>0$ and $\mathfrak c > 0$, then there is a unique $q_{\lim}(g,\sigma,r)=\mathfrak c>0$ such that $\kapparq\leq 0$ for $q \geq q_{\lim}(g,\sigma,r)$ and $\kapparq> 0$ for $q < q_{\lim}(g,\sigma,r)$.
\item If $\m\leq 0$, we choose $q_{\lim}(g,\sigma,r,q)=r$.
\end{enumerate}
We conclude the proof of {\it i)} using Corollary \ref{Cor_CNS_ext_ps}.

\textbf{$ii)$}
Let $\eta:=(g-\sigma^2)/(2\E[\ln(1/\Theta)])$.
\begin{enumerate}
\item First assume that $ \eta<q. $ In this case, if
$ q<r$, we get $\m<0$ and $\mathfrak{A}_x(g,\sigma,r,q)= \infty$ according to Corollary~\ref{Cor_CNS_ext_ps}.
 Moreover, $ q\geq r$ implies $\mathfrak{A}_x(g,\sigma,r,q)= 0$ according to Corollary \ref{Cor_CNS_ext_ps}.
 We thus choose $r_{\lim}(g,\sigma,q)=q$ in this case.

\item Next, assume that 
$ q \leq \eta. $
Then, 

\noindent
- If $r \leq q$, we obtain $\mathfrak{A}_x(g,\sigma,r,q)= 0$ according to Corollary \ref{Cor_CNS_ext_ps}.

\noindent
- If $r > \eta $, then $r>q$ and  $\m\leq 0$. According to Corollary~\ref{Cor_CNS_ext_ps}, $\mathfrak{A}_x(g,\sigma,r,q)= \infty$.

\noindent
- If $q < r =\eta$, then $\m=0$. According to Corollary~\ref{Cor_CNS_ext_ps}, $\mathfrak{A}_x(g,\sigma,r,q)= \infty$.

\noindent
- If $q < r<\eta$, then $\m>0$. Thus, the value of $\mathfrak{A}_x(g,\sigma,r,q)$ depends on the sign of $\kapparq$.
From \eqref{defLaplexp}, we see that for any $\lambda \in \R$ the value of $\phi(\lambda)$ increases when $r$ increases. This implies that it is also the case for the value of $\phi(\hat{\tau})$, by definition of $\hat{\tau}$ as the argument of the minimum of $\phi$. As a consequence $\kapparq=\phi(\hat{\tau})+r-q$ is strictly increasing with $r$. Next, for $r=q$,
$$ \kapparq=\phi(\hat{\tau})+q-q <0,$$
and, when $r$ tends to $\eta$, $\hat{\tau}$ tends to $0$ and thus, $\phi(\hat\tau)$ tends to $0$. As $r>q$, $\kapparq$ tends to a positive value as $r$ tends to $\eta$.
We deduce that there exists a unique $r_{\lim}(g,\sigma,q)  \in (q, \eta)$ such that $\kapparq=0$  and that $\kapparq<0$ when $r<r_{\lim}(g,\sigma,q)$ (resp. $\kapparq>0$ when $r>r_{\lim}(g,\sigma,q)$). We conclude again by an application of Corollary \ref{Cor_CNS_ext_ps}.
\end{enumerate}

\textbf{$iii)$}
\begin{enumerate}
\item If $r \leq q$, $\mathfrak{A}_x(g,\sigma,r,q)= 0$ and $g_{\lim}(\sigma,r,q)=0$ satisfies the needed property.
\item If $q<r$ and $g \leq \sigma^2 - 2r \E[\ln \Theta]$, then $\m \leq 0$ and  $\mathfrak{A}_x(g,\sigma,r,q)= \infty$.
\item If $q<r$ and $g > \sigma^2 - 2r \E[\ln \Theta]$, then $\m > 0$ and we have to study the sign of $\kapparq$.
As $r>q$, $\kapparq$ is positive when $g$ tends to $\sigma^2 - 2r \E[\ln \Theta]$, because $\hat\tau$ tends to $0$. Moreover, from the equality
\begin{equation}\label{equal_g} g = \sigma^2 - 2 \hat\tau \sigma^2 - 2r \E[\Theta^{\hat{\tau}}\ln \Theta] \end{equation}
characterizing $\hat{\tau}$,  we see that $\hat{\tau}$ goes to $-\infty$ when $g$ goes to $+\infty$. Combining the definition of $\kapparq$ in \eqref{eq:kapparq} with \eqref{equal_g}, we obtain
\begin{align*}  \kapparq&= \frac{\hat{\tau}}{2}(g-\sigma^2) + r \left[ \E[\Theta^{\hat{\tau}}(2-\ln \Theta^{\hat{\tau}})] -1\right]- q\\
& \leq \frac{\hat{\tau}}{2}(g-\sigma^2) + r  \left[ \E[\Theta^{\hat{\tau}}](2-\ln \E[\Theta^{\hat{\tau}}]) -1\right]- q 
\end{align*}
where the inequality is a consequence of the concavity of the function $x \mapsto x(2-\ln x)$. Therefore, when $g$ goes to $\infty$, $\hat{\tau}$ goes to $-\infty$, and $\kapparq$ goes to $-\infty$. Moreover, we have
\begin{align*} \frac{d\kapparq}{dg}&= \frac{d \hat{\tau}}{2dg}(g-\sigma^2)+\frac{\hat{\tau}}{2} + r  \frac{d\E[\Theta^{\hat{\tau}}(2-\ln \Theta^{\hat{\tau}})]}{d \hat{\tau} }  \frac{d \hat{\tau}}{dg} \\
&= \frac{d \hat{\tau}}{2dg}(g-\sigma^2)+\frac{\hat{\tau}}{2} + r  \E[\Theta^{\hat{\tau}}\ln \Theta(1-\ln \Theta^{\hat{\tau}})] \frac{d \hat{\tau}}{dg},
\end{align*}
which is negative as a sum of negative terms, as from \eqref{equal_g}, we see that $\hat{\tau}$ decreases when $g$ increases. Thus, $\kapparq$ is decreasing with $g$. We deduce that there exists a unique $g_{\lim}(\sigma,r,q)  \in (\sigma^2-2r\E[\ln \Theta],\infty)$ such that $\kapparq=0$,  and that $\kapparq>0$ when $g<g_{\lim}(\sigma,r,q)$ (resp. $\kapparq<0$ when $g>g_{\lim}(\sigma,r,q)$). We conclude again by an application of Corollary \ref{Cor_CNS_ext_ps}.

\end{enumerate}

\textbf{$iv)$}
\begin{enumerate}
\item If $r \leq q$, $\mathfrak{A}_x(g,\sigma,r,q)= 0$ and $\sigma_{\lim}(g,r,q)=\infty$ satisfies the needed property.
\item If $q<r$ and $g \leq  - 2r \E[\ln \Theta]$, then $\m \leq 0$ and  $\mathfrak{A}_x(g,\sigma,r,q)= \infty$, and $\sigma_{\lim}(g,r,q)=0$ satisfies the needed property.
\item If $q<r$ and $g >  - 2r \E[\ln \Theta]$
\begin{itemize}
\item If $ \sigma^2 \geq  g+ 2r \E[\ln \Theta]$, then $\m \leq 0$ and  $\mathfrak{A}_x(g,\sigma,r,q)= \infty$.
\item If $ \sigma^2 <  g+ 2r \E[\ln \Theta]$, then $\m > 0$ and we have to study the sign of $\kapparq$. From \eqref{defLaplexp} and \eqref{eq:kapparq}, we have
$$
\frac{\partial\kapparq}{\partial\sigma^2} = \frac{\partial\hat\tau}{\partial\sigma^2}\phi'(\hat\tau) + \hat\tau(\hat\tau - 1)=\hat\tau(\hat\tau - 1),
$$
where we used that $\phi'(\hat\tau)=0$ because $\hat\tau$ is the argument of the minimum of $\phi$. As $\hat\tau<0$, we obtain that
$\kapparq$ is increasing with $\sigma^2$. Moreover, when $\sigma^2$ tends to $g+ 2r \E[\ln \Theta]$, 
$$ \kapparq \longrightarrow \hat{\tau}^2 g + r-q +2r \E[\Theta^{\hat{\tau}}-1+(\hat\tau - 1)\ln \Theta^{\hat{\tau}}]>0.   $$
For $\sigma^2=0$, combining \eqref{equal_g} and the definition of $\kapparq$, we have
$$
\kapparq=2r\E[\Theta^{\hat{\tau}}(1 - \ln(\Theta^{\hat\tau}))] - r-q,
$$
which can be positive or negative, depending on $g,r$ and the law of $\Theta$. Then, if for $\sigma^2=0$, $\kapparq<0$, there exists $\sigma_{\lim}>0$ such that $\kapparq=0$ for $\sigma=\sigma_{\lim}$. Else, $\kapparq>0$ for all $\sigma\geq 0$ and $\mathfrak{A}_x(g,\sigma,r,q)=\infty$.
\end{itemize}
We conclude as for the previous points.
\end{enumerate}
\end{proof}

\subsection{Proof of Section \ref{sec:role-partitioning}}
We now explore how the long time behaviour of the infection depends on 
the parasites repartition kernel. 
We focus in particular on the uniform and the equal sharing, two cases where explicit computations are doable.
 
\begin{proof}[Proof of Corollary \ref{Cor_CNS_ext_ps_unif}]
We first consider $\kappa(d\theta)=d\theta$. We get $\lambda^-=-1$ and for $\lambda>-1$,
$$\laplexp(\lambda)=
%\lambda g + 2 r \left[\int_0^1 \theta^\lambda d\theta -1\right]=
\lambda g+ 2 r \left[\frac{1}{\lambda+1} -1\right],\quad \laplexp'(\lambda)=
g - 2 r \frac{1}{(\lambda +1)^2}, \quad \m=
g+ 2r\int_0^1 \ln\theta d\theta = g- 2r.$$

The minimum of $\laplexp$ on $(-1,\infty)$ is reached at 
$ \hat{\tau} = \sqrt{2rg^{-1}}-1 $
and equals
$$ \laplexp(\hat{\tau})=\left(\sqrt{\frac{2r}{g}}-1\right) g + 2 r \left[\sqrt{\frac{g}{2r}} -1\right] = 2 \sqrt{2rg}-g-2r.  $$
Let us look at the sign of $\laplexp(\hat{\tau}) +r-q = 2\sqrt{2rg}-g-r-q.$ This quantity is nonpositive if and only if
$
8rg\leq g^2+ (r+q)^2+ 2g(r+q).
$
Therefore, setting $X = g$, we have to solve the second degree polynomial equation
$$
X^2 +2X(q-3r)+(r+q)^2= 0. 
$$
Recall that $r>q$. In this case, the two solutions are given by 
$$X_1 = 3r-q-2\sqrt{2r(r-q)},\ X_2 = 3r-q+2\sqrt{2r(r-q)},$$
so that $\laplexp(\hat{\tau}) +r-q $ is negative for $g<X_1$ or $g> X_2$.
Notice that $X_1-2r =r-q-2\sqrt{2r(r-q)}=\sqrt{r-q}(\sqrt{r-q}-2\sqrt{2r})<0$ and $X_2>2r$. Then, the condition $(\m>0 \text{ and } \laplexp(\hat\tau) +r-q\leq 0)$ is equivalent to
$g\geq 3r-q+2\sqrt{2r(r-q)},
$
and using Corollary \ref{Cor_CNS_ext_ps} \ref{it:1-23}, \al{we proved that $g_{\text{lim}}\leq 3r-q+2\sqrt{2r(r-q)}$}.

\al{We now prove that $g_{\text{lim}}\geq 3r-q+2\sqrt{2r(r-q)}$.} If $g< 3r-q+2\sqrt{2r(r-q)}$, we distinguish two cases: if $g\leq 2r$ then $\m\leq 0$ and if 
$ 2r<g< 3r-q+2\sqrt{2r(r-q)}$, then $(\m>0 \text{ and } \laplexp(\hat\tau) +r-q> 0)$
so that using Corollary \ref{Cor_CNS_ext_ps} \ref{it:2-23}, \al{we get the result}.

 Let us now consider the case where the cells share equally their parasites between their two daughters ($\Theta \equiv 1/2$).
In this case we have $\lambda^-=-\infty$ and for $\lambda \in \R$,
$$\laplexp(\lambda)=\lambda g + 2 r \left[2^{-\lambda} -1\right],\quad \laplexp'(\lambda)= g-  2^{1-\lambda} r \ln 2,\quad \m= g -  2r\ln 2 .$$
 The minimum of $\laplexp$ on $\R$ is reached at 
$ \hat{\tau}= (\ln 2)^{-1} \ln \left( 2r \ln 2g^{-1} \right) $
and
$ \laplexp(\hat{\tau})= g \hat{\tau} + g(\ln 2)^{-1}-2r.$
Thus to have almost sure extinction of the cell population, the two following conditions must be satisfied: 
\begin{align*}
% \label{cond_sym}
 2r\ln 2<g \quad \text{and} \quad \frac{g}{r\ln 2} \left(1+\ln 2- \ln \left( \frac{g}{r\ln 2}  \right)\right)-1-\frac{q}{r} \leq 0.
\end{align*}
Let 
$$\varphi(x) = x \left(1+\ln 2- \ln \left( x  \right)\right)-1-\frac{q}{r}.$$
We are looking for the sign of $\varphi$ on $\al{(2,+\infty)}$, interval on which the first condition $\m>0$ is satisfied. On this interval, $\varphi$ is decreasing from $1-q/r>0$ to $-\infty$. Thus, there exists $x_0(q/r)>2$ such that $\varphi(x_0(q/r))=0$ and
\begin{align*}
\text{if } 2r\ln(2)<g< rx_0(q/r)\ln(2),&\quad \text{then } (\m>0\text{ and } \laplexp(\hat{\tau}) +r-q>0),\\
\text{if }g\geq  rx_0(q/r)\ln(2),&\quad \text{then } (\m>0\text{ and } \laplexp(\hat{\tau}) +r-q\leq 0).
\end{align*}
Finally, applying Corollary \ref{Cor_CNS_ext_ps}, we get
$$
\begin{array}{lll}
\text{if } & g\geq rx_0(q/r)\ln(2), & \text{ then } \lim_{t \to \infty} \E[\mathfrak{C}_t]=0,\\
\text{if }2r\ln(2)<&g<rx_0(q/r)\ln(2)\text{ or } g\leq 2r\ln 2, & \text{ then }  \lim_{t \to \infty} \E[\mathfrak{C}_t]=\infty,\\
\end{array}
$$
which yields the result.

\end{proof}

\begin{proof}[Proof of Proposition \ref{prop:inf}]

Let $\kappa$ be a partitioning kernel and $\Theta$ a random variable with distribution $\kappa$. Then, let us define $\mathfrak{m}(y)=y+2\E[\ln\Theta]$. For any $y\leq -2\E[\ln\Theta]$, $\mathfrak{m}(y)\leq 0$ so that by Corollary \ref{Cor_CNS_ext_ps}, $\lim_{t\rightarrow+\infty}\E_{\delta_x}[\mathfrak{C}_t]=\infty$. By Jensen's inequality, $-\E[\ln\Theta]>\ln 2$. Then, for $g/r<2\ln2$, we proved the result. 

For all $y>2\ln 2$, let us define $\hat{\tau}(y)$ as the solution of 
\begin{align}\label{eq:deftauy}
2\E[\ln(\Theta)\Theta^{\hat{\tau}(y)}]=-y.
\end{align}

Let $y_0>2\ln 2$ be such that 
\begin{align}\label{eq:kapparqy0}
y_0\hat{\tau}(y_0) + 2\E[\Theta^{\hat{\tau}(y_0)}]-1-q/r=0.
\end{align}
We want to prove that $y_0>~\ln(2)x_0(q/r)$, where $x_0(q/r)$ is defined in \eqref{eq:x0rq}. By definition, $x_0$ is such that $\varphi(x_0)=x_0(1+\ln(2)-\ln(x_0))-1-q/r=0$. As $\varphi$ is a decreasing function on $(2,+\infty)$ from $1-q/r$ to $-\infty$, we proved that $y_0>~\ln(2)x_0(q/r)$ if $\varphi(y_0/\ln(2))<\varphi(x_0)=0$. Therefore, we need to prove that
that $y_0/\ln(2)(1+\ln(2\ln(2))-\ln(y_0))-1-q/r<0$. Combining the latter with \eqref{eq:kapparqy0}, we need to prove that 
\begin{align*}
y_0/\ln(2)(1+\ln(2\ln(2)-\ln(y_0))-y_0\hat{\tau}(y_0) - 2\E[\Theta^{\hat{\tau}(y_0)}]<0.
\end{align*}

For $y>2\ln2$, let us define $F(y)=y/\ln(2)(1+\ln(2\ln(2)) -\ln(y))-y\hat{\tau}(y) - 2\E[\Theta^{\hat{\tau}(y)}]$. We have
\begin{align*}
F'(y) = \frac{1}{\ln(2)}(1+\ln(2\ln(2)) -\ln(y))-\frac{1}{\ln(2)}-\hat{\tau}(y)-\frac{d\hat{\tau}}{dy}(y)\left(y +  2\E[\ln(\Theta)\Theta^{\hat{\tau}(y)}]\right).
\end{align*}
By definition of $\hat{\tau}(y)$ in \eqref{eq:deftauy}, we obtain
\begin{align*}
F'(y) = \frac{\ln(2\ln(2))-\ln(y)}{\ln(2)}-\hat{\tau}(y).
\end{align*}
Next, $x\mapsto \ln(x)x^\tau$ is concave for $\tau<0$, and $x\in(0,1)$. Therefore, 
$$
\E[\ln(\Theta)\Theta^{\hat{\tau}(y)}]\leq \frac{-\ln(2)}{2^{\hat{\tau}(y)}},
$$
so that by \eqref{eq:deftauy} again, $y\geq 2\ln(2)2^{-\hat{\tau}(y)}$ and 
$$
\hat{\tau}(y)\geq\frac{\ln(2\ln(2))-\ln(y)}{\ln(2)},
$$
and $F$ is decreasing on $(2\ln(2),\infty)$. Moreover, if $X=-\hat{\tau}(2\ln(2))$,
$$
F(2\ln(2))=2+2\ln(2)X-2\E[\Theta^{-X}]\leq 2(1 +\ln(2)X-2^{X})\leq 0,
$$
where we used again Jensen's inequality. Then $F(y)\leq 0$ for all $y\geq 2\ln(2)$, and in particular, $F(y_0)<0$, which ends the proof.
\end{proof}

\begin{proof}[Proof of Proposition \ref{prop:multimodal}]
Let $\vartheta \in (0,1/2)$ and $\Theta_\vartheta$ be a random variable with distribution 
$$\kappa_\vartheta(d\theta)=\frac{1}{2}\left(\delta_\vartheta(d\theta) + \delta_{1-\vartheta}(d\theta)\right).$$
Let $\Theta$ be a symmetric random variable on $(0,1)$  with distribution $\kappa$ and with expectation $1/2$, such that
$$ \E[\min(\Theta , (1-\Theta))]=\vartheta. $$
Let
\begin{align*}
y_\vartheta^\star=\sup\left\{y\geq 0\text{ s.t. if }g/r<y,\forall x\geq 0, \lim_{t\rightarrow+\infty}\E^{(\kappa_\vartheta)}_{\delta_x}[\mathfrak{C}_t]=\infty\right\}. 
\end{align*}
and
\begin{align*}
y^\star=\sup\left\{y\geq 0\text{ s.t. if }g/r<y,\forall x\geq 0, \lim_{t\rightarrow+\infty}\E^{(\kappa)}_{\delta_x}[\mathfrak{C}_t]=\infty\right\}. 
\end{align*}
Let us define
$$\mathfrak{m}(y):=y + 2\E[\ln\Theta]= y + 2\E[\mathbf{1}_{\Theta \leq 1/2}\ln(\Theta(1-\Theta))]$$
and
$$y_0:= - 2\E[\ln(\Theta)]=- 2\E[\mathbf{1}_{\Theta \leq 1/2}\ln(\Theta(1-\Theta))],$$
where the two rewritings of the expectation are a consequence of the symmetry \al{with respect to $1/2$} of the random variable $\Theta$. 
If $y\leq y_0$, then $\mathfrak{m}(y)\leq 0$ so that $y^\star\geq y_0$ by Corollary~\ref{Cor_CNS_ext_ps}. Similarly, $y_\vartheta^\star\geq-\ln(\vartheta(1-\vartheta))$.

First, if $y_\vartheta^\star<y_0$, then , $y_\vartheta^\star<y^\star$ and the result is proved. Next, assume that $y_\vartheta^\star\geq y_0$. For all $y\geq y_0$, let $\hat{\tau}(y),\hat{\tau}_\vartheta(y)\leq 0$ be such that
\begin{align}\label{eq:deftauhatmy}
2\E\left[\Theta^{\hat{\tau}(y)}\ln\Theta\right] = -y,\quad 2\E\left[\Theta_\vartheta^{\hat{\tau}_\vartheta(y)}\ln\Theta_\vartheta\right]=-y.
\end{align}  
Then, by definition of $y^\star$, according to Corollary \ref{Cor_CNS_ext_ps}, 
\begin{align}\label{eq:dm}
y^\star\hat{\tau}(y^\star)+2\E[\mathbf{1}_{\Theta \leq 1/2}(\Theta^{\hat{\tau}(y^\star)} + (1-\Theta)^{\hat{\tau}(y^\star)})] = 1+q/r.
\end{align}
Next, for all $y\geq y_0$, let $\varphi_\vartheta(y)=y\hat{\tau}_{\vartheta}(y) + \vartheta^{\hat{\tau}_{\vartheta}(y)} + (1-\vartheta)^{\hat{\tau}_{\vartheta}(y)}-1-q/r.$ Using \eqref{eq:deftauhatmy}, we obtain that $\varphi_\vartheta'(y)=\hat{\tau}_\vartheta(y)<0$, so that $\varphi_\vartheta$ is decreasing on $[y_0,+\infty)$. Moreover, $\varphi_\vartheta(y_\vartheta^\star)=0$. Therefore, to show that $y^\star\geq y_\vartheta^\star$, we need to prove that $\varphi_\vartheta(y^\star)\leq 0$. Combining the definition of $\varphi_\vartheta$ with \eqref{eq:dm}, we obtain
\begin{align*}
\varphi_\vartheta(y^\star) & = y^\star\hat{\tau}_{\vartheta}(y^\star) + \vartheta^{\hat{\tau}_{\vartheta}(y^\star)} + (1-\vartheta)^{\hat{\tau}_{\vartheta}(y^\star)}-y^\star\hat{\tau}(y^\star)-2 \E[\mathbf{1}_{\Theta \leq 1/2} (\Theta^{\hat{\tau}(y^\star)}+(1-\Theta)^{\hat{\tau}(y^\star)} )]
\\
& = \varphi_\vartheta(y^\star) - \varphi(y^\star),
\end{align*}
where for all $y\geq  y_0$, $\varphi(y) = y\hat{\tau}(y)+2\E[\mathbf{1}_{\Theta \leq 1/2} (\Theta^{\hat{\tau}(y)}+(1-\Theta)^{\hat{\tau}(y)} )]- 1-q/r$.

To prove that $\varphi_\vartheta(y^\star)$ is negative, let us define $F(y)=\varphi_\vartheta(y) - \varphi(y)$,  for all $y\geq y_0$. Using \eqref{eq:deftauhatmy}, we have $F'(y)=\hat{\tau}_\vartheta(y) -\hat{\tau}(y)$. To find the sign of this latter, we study the convexity of $\varphi_1:z\in (0,1)\mapsto \ln(z) z^\tau +\ln(1-z)(1-z)^\tau$, for any $\tau<0$. We have
\begin{align*}
&\varphi_1'(z) = z^{\tau-1}\left(1 + \tau\ln(z)\right) -(1-z)^{\tau-1}\left(1 + \tau\ln(1-z)\right)\\
&\varphi_1''(z) = z^{\tau-2}\left((\tau-1)\left(1 + \tau\ln(z)\right)+\tau\right) + (1-z)^{\tau-2}\left((\tau-1)\left(1 + \tau\ln(1-z)\right)+\tau\right)<0,
\end{align*}
as $\tau<0$ and $z\in(0,1)$. Therefore, $\varphi_1$ is concave on $(0,1)$. Then, for all $y\geq y_0$, by Jensen's inequality we obtain
$$
2\E[\mathbf{1}_{\Theta \leq 1/2} (\ln \Theta\Theta^{\hat{\tau}(y)}+\ln (1-\Theta)(1-\Theta)^{\hat{\tau}(y)} )]\leq \ln(\vartheta)\vartheta^{\hat{\tau}(y)} +\ln(1-\vartheta)(1-\vartheta)^{\hat{\tau}(y)},
$$
where we used that $\vartheta =2\E[\mathbf{1}_{\Theta \leq 1/2} \Theta] $. Combining the last inequality with \eqref{eq:deftauhatmy}, we have
$$
\ln(\vartheta)\vartheta^{\hat{\tau}_\vartheta(y)} + \ln(1-\vartheta)(1-\vartheta)^{\hat{\tau}_\vartheta(y)}\leq \ln(\vartheta)\vartheta^{\hat{\tau}(y)} +\ln(1-\vartheta)(1-\vartheta)^{\hat{\tau}(y)}.
$$
As $\tau\mapsto \ln(z)z^\tau +\ln(1-z)(1-z)^\tau$ is non-decreasing on $\mathbb{R}_-$, we get that 
\begin{align}\label{eq:taumtautheta}
\hat{\tau}_\vartheta(y)\leq \hat{\tau}(y),\ \text{for all }y\geq y_0.
\end{align}
Therefore, $F$ is non-increasing on $[y_0,\infty)$. Finally, as $F(y^\star)=\varphi_\vartheta(y^\star)\leq F(y_0)$, we will now prove that $F(y_0)\leq 0$.  We have 
$$
F(y_0) = y_0\left(\hat{\tau}_{\vartheta}(y_0)-\hat{\tau}(y_0)\right) + \vartheta^{\hat{\tau}_{\vartheta}(y_0)} + (1-\vartheta)^{\hat{\tau}_{\vartheta}(y_0)}-2\E[\mathbf{1}_{\Theta \leq 1/2} (\Theta^{\hat{\tau}(y_0)}+(1-\Theta)^{\hat{\tau}(y_0)} )].
$$
As $z\mapsto z^\tau + (1-z)^\tau$ is convex on $(0,1)$ for any $\tau<0$, we have by Jensen's inequality
$$
2\E[\mathbf{1}_{\Theta \leq 1/2} (\Theta^{\hat{\tau}(y_0)}+(1-\Theta)^{\hat{\tau}(y_0)} )]\geq \vartheta^{\hat{\tau}(y_0)} + (1-\vartheta)^{\hat{\tau}(y_0)}.
$$
Then,
$$
F(y_0)\leq y_0\left(\tau_\vartheta - \tau\right) + \psi(\tau_\vartheta) - \psi(\tau),
$$
where $\tau_\vartheta = \hat{\tau}_{\vartheta}(y_0)$, $\tau = \hat{\tau}(y_0)$, and $\psi(\tau)= \vartheta^\tau +  (1-\vartheta)^\tau$. Recall that $\tau_\vartheta\leq \tau$ according to \eqref{eq:taumtautheta}. By Taylor formula with integral remainder, we have
$$
\psi(\tau)-\psi(\tau_\vartheta) = \psi'(\tau_\vartheta)(\tau - \tau_\vartheta) + \int_{\tau_\vartheta}^{\tau}\psi''(z)(\tau - z)dz.
$$ 
First, note that $\psi'(\tau_\vartheta) = \ln\vartheta\vartheta^{\tau_\vartheta} + \ln(1-\vartheta)(1-\vartheta)^{\tau_\vartheta} = -y_0$ according to \eqref{eq:deftauhatmy}. Then, 
$$
F(y_0) \leq -\int_{\tau_\vartheta}^{\tau}\psi''(z)(\tau - z)dz\leq 0,
$$
which ends the proof.
\end{proof}

\begin{proof}[Proof of Proposition \ref{prop_bigg}]
Let us first assume that 
$$ g \leq -  r \ln \mintheta(1-\mintheta) $$
In this case, we choose $p_1=0, p_2 = 1/2$ and $z_2 = \mintheta$. This choice entails
$$ \m = g +  r \ln \mintheta(1-\mintheta)  \leq 0  $$
and we conclude by an application of Corollary \ref{Cor_CNS_ext_ps}.

Let us now assume that 
$$ g > -  r \ln \mintheta(1-\mintheta). $$
In this case, we choose 
$$z_1 \in (0,\mintheta/4), \ p_1= z_1^{(\ln \ln (1/z_1))^{-1}} , p_2 = 1/2-z_1^{(\ln \ln (1/z_1))^{-1}}$$ and 
$$z_2 = \frac{\mintheta-2 z_1^{1+(\ln \ln (1/z_1))^{-1}}}{1-2 z_1^{(\ln \ln (1/z_1))^{-1}}}$$
This choice entails
$$  \E\left[\min(\Theta, 1-\Theta)\right]=\mintheta $$
and
\begin{multline*} \m(z_1) = g +  2r z_1^{(\ln \ln (1/z_1))^{-1}}\ln (z_1(1-z_1))  \\ +2r \left( \frac{1}{2}- z_1^{(\ln \ln (1/z_1))^{-1}} \right) \ln\left( \frac{\mintheta-2 z_1^{1+(\ln \ln (1/z_1))^{-1}}}{1-2 z_1^{(\ln \ln (1/z_1))^{-1}}}\frac{1-\mintheta-(1-z_1)2 z_1^{(\ln \ln (1/z_1))^{-1}}}{1-2 z_1^{(\ln \ln (1/z_1))^{-1}}} \right) \end{multline*}
where we indicated explicitely the dependence on $z_1$ for the sake of readability. We will make $z_1$ tend to $0$ and prove that for $z_1$ small enough the distribution of $\Theta$ meets the required properties. First we notice that
$$ \lim_{z_1 \to 0}  z_1^{(\ln \ln (1/z_1))^{-1}}\ln (z_1 (1-z_1))=0 $$
and
\begin{align*}
 \lim_{z_1 \to 0}&  \left( \frac{1}{2}- z_1^{(\ln \ln (1/z_1))^{-1}} \right) \ln\left( \frac{\mintheta-2 z_1^{1+(\ln \ln (1/z_1))^{-1}}}{1-2 z_1^{(\ln \ln (1/z_1))^{-1}}} \frac{1-\mintheta-(1-z_1)2 z_1^{(\ln \ln (1/z_1))^{-1}}}{1-2 z_1^{(\ln \ln (1/z_1))^{-1}}} \right)\\
& =\frac{\ln \mintheta(1-\mintheta)}{2}.
\end{align*}
We deduce that 
$$ \lim_{z_1 \to 0}\m(z_1)= g +  r \ln \mintheta(1-\mintheta)>0 $$
and thus there exists $\mathfrak{z}>0$ such that for all $z_1<\mathfrak{z}$, $\m(z_1)>0$. Therefore, according to Corollary \ref{Cor_CNS_ext_ps}, to prove the result, we need to prove that there exists $z_1<\mathfrak{z}$ such that $\kapparq(z_1)>0$, \al{where $\kapparq(z_1)$ is the constant defined in \eqref{eq:kapparq} in the case $\kappa(d\theta)=\sum_{i=1}^2\left(\delta_{z_i}(d\theta)+\delta_{1-z_i}(d\theta)\right)p_i$ with $z_1,z_2, p_1,p_2$ defined above}. 

First, we know that for all $z_1<\mathfrak{z}$, as $\m(z_1)>0$, the argument of the minimum of $\phi$ is negative, {\it i.e.} there exists $\tau(z_1)<0$ such that
\begin{align}\label{eq:tauz1}
g= -2r \E[\Theta^{\tau(z_1)}\ln \Theta].
\end{align}

We now prove by a reductio ad absurdum that for all $z_1<\mathfrak{z}\wedge e^{-g/r}$,
\begin{equation} \label{small_tau} \tau(z_1)> - \frac{1}{\ln \ln (1/z_1)}. \end{equation}
Let $z_1<\mathfrak{z}\wedge e^{-g/r}$. If $\tau(z_1)\leq - 1/\ln \ln (1/z_1)$, we have
\begin{align*} 
\frac{g}{r}&\geq p_1 z_1^{\tau(z_1)}\ln (1/z_1)  = z_1^{(\ln \ln (1/z_1))^{-1}}z_1^{\tau(z_1)}\ln (1/z_1)\\
&  \geq z_1^{(\ln \ln (1/z_1))^{-1}}z_1^{- (\ln \ln (1/z_1))^{-1}}\ln (1/z_1)=\ln (1/z_1),
\end{align*}
so that $z_1\geq e^{-g/r}$ which is absurd.
We deduce that \eqref{small_tau} holds for all $z_1<\mathfrak{z}\wedge e^{-g/r}$, and as $\tau(z_1)<0$, we obtain that $\tau(z_1)$ goes to $0$ when $z_1$ goes to $0$.

Finally, with our choice of parameters we have
$$ \kapparq(z_1) = \tau(z_1) g +2r \left( p_1\left(z_1^{\tau(z_1)}+(1-z_1)^{\tau(z_1)}\right)+p_2\left(z_2^{\tau(z_1)}+(1-z_2)^{\tau(z_1)}\right) \right)-r-q, $$
and
$$ \lim_{z_1 \to 0}\tau(z_1)= \lim_{z_1 \to 0}p_1=0, \quad \lim_{z_1 \to 0}(1-z_1)^{\tau(z_1)}=1, \quad \lim_{z_1 \to 0}p_2=\frac{1}{2} \quad \text{and} \quad \lim_{z_1 \to 0} z_2 = \mintheta.  $$
Moreover, by \eqref{eq:tauz1},
$$
\lim_{z_1\to 0}z_1^{\tau(z_1)}\ln(z_1)p_1 = -\frac{g}{2r} + \frac{1}{2}\ln(\vartheta(1-\vartheta)),
$$
so that $\lim_{z_1\to 0}z_1^{\tau(z_1)}p_1 = 0$. Finally,
$$  \lim_{z_1 \to 0} \kapparq(z_1)=r-q>0. $$
This ends the proof.
\end{proof}
\subsection{Proof of Section \ref{section_beta}}
We now turn to the proof of the results on the asymptotic behaviour of the number of parasites in the cells. Hence, we consider 
that the dynamics of the parasites follows \eqref{X_sans_sauts_stables}.
\begin{proof}[Proof of Proposition \ref{prop:beta_ct:g_var}]
From Section \ref{sec:MTO}, we know that the auxiliary process $Y$ is the unique strong solution to the SDE
\begin{align*}
Y_t= x& + \int_0^tg(Y_s)ds+ \int_0^t \sqrt{2 \sigma^2(Y_s) }dB_s+ \int_0^t \int_0^{p(Y_{s^-})}
\int_{\mathbb{R}_+}z\widetilde{Q}(ds,dx,dz)\\
&+\int_0^t \int_0^1  (\Theta(Y_{s^-},\zeta)-1)Y_{s^-}N(ds,d\zeta), \nonumber
\end{align*}
where $N$ is as in \eqref{X_sans_sauts4}.
Let us begin with the proof of point \ref{it:2-31}. Note that as $g(x)/x+2r\E[\ln\Theta(x)]<-\eta$ for all $x>0$, \ref{SNinfinity} is satisfied. We plan to apply (6.3) of \cite[Theorem 6.2]{companion}. This result still holds with $\Theta(x)$ instead of $\Theta$. Indeed the proof of this result \cha{requires} two properties on the partitioning kernel. First, we need that $Y_t\exp(-\int_0^t g(Y_s)/Y_sds- \int_0^t \int_0^1\ln \Theta(Y_{s^-},\zeta)N(ds,d\zeta))$ is a local martingale, which still holds. Second, we need a (possibly stochastic) lower bound on the proportion $\Theta(x)$ of the number of parasites that goes to one of the daughter cells at division, uniform in $x$. This is ensured by our assumptions.
Thus, from (6.3) of \cite[Theorem 6.2]{companion}, we have
\begin{equation*} \lim_{t\rightarrow +\infty} Y_t = 0 \quad \text{almost surely,} 
\end{equation*}
and combining \eqref{eq:mto} with the fact that $\E_{\delta_x}\left[N_t\right]=e^{(r-q) t}$, we obtain that 
\begin{equation*} 
\E_{\delta_x} \left[ \frac{\sum_{u \in V_t} \mathbf{1}_{\{X_t^u >\eps\}} }{e^{(r-q) t}}\right] \xrightarrow[t\rightarrow \infty]{} 0.
\end{equation*}
Moreover, the fact that $(N_t,t\geq 0)$ is a birth and death process with individual death rate $q$ and individual birth 
rate $r$ also entails that $N_te^{-(r-q) t}$ converges in probability to an exponential random 
variable with parameter $1$ on the event of survival, when $t$ goes to infinity. Hence, we have
\begin{align*}
\mathbf{1}_{\{N_t \geq 1\}}\frac{\sum_{u \in V_t} \mathbf{1}_{\{X_t^{u} >\eps\}} }{N_t} = \frac{\sum_{u \in V_t} \mathbf{1}_{\{X_t^{u} >\eps\}} }{e^{(r-q) t}}\times \frac{\mathbf{1}_{\{N_t \geq 1\}}}{N_te^{-(r-q) t}}\to 0 \quad \text{in probability}, \quad \cha{(t\rightarrow \infty)}.
\end{align*}
It ends the proof of point \ref{it:2-31}.\\

We now prove point \ref{it:3-31}. Applying again (6.3) of \cite[Theorem 6.2]{companion} to $Y$, we obtain that
$$\mathbb{P}\left( Y_t\neq0\right)\rightarrow 0,\quad (t\rightarrow \infty).$$
From this, similarly as for the proof of point {\it ii)} we obtain that
\begin{equation*} \label{conv_proba1}  \mathbf{1}_{\{N_t\geq 1\}}\frac{\sum_{u \in V_t} \mathbf{1}_{\{X_t^u >0\}} }{N_t} \to 0 \quad \text{in probability}, \quad (t\rightarrow \infty). \end{equation*}

To end the proof of point \ref{it:3-31}, we need to prove that the aforementioned convergence holds almost surely. 
We cannot follow directly the proof of \cite[Theorem 4.2(i)]{BT11} because their Lemma 4.3 concerns Yule processes and does not hold when we take into account the death of cells. However, we can 
prove a result similar to this lemma (see Lemma \ref{lemtechnq} in the Appendix) which is sufficient to get our result.
Except from this lemma the proof is exactly the same and we thus refer to \cite{BT11} for details of the proof.\\

We end with the proof of point \ref{it:1-31}.
Applying  \cite[Corollary 6.4.{\it iii)}]{companion} to $Y$, we obtain that
\begin{equation}\label{Wpos} \liminf_{t \to \infty} Y_t e^{-\Lambda_t} = W, \end{equation}
with $\P(W>0)>0$ and where $\Lambda$ is defined by
$$ \Lambda_t:= \int_0^t \frac{g(Y_s)}{Y_s}ds + \int_0^t \int_0^1  \ln \Theta(Y_{s^-},\zeta)N(ds,d\zeta), $$
where $N$ is PPM on $\R_+\times[0,1]$ with intensity $2rds\otimes d\zeta$.
Notice that $\Lambda$ may be rewritten as
\begin{align*}
\Lambda_t &= \int_0^t\left( \frac{g(Y_s)}{Y_s} + 2r\E \left[  \ln \Theta(Y_{s^-}) \right] \right)ds + \cha{\int_0^t\int_0^1  \left( \ln \Theta(Y_{s^-},\zeta)N(ds,d\zeta)- 2r\E \left[  \ln\Theta(Y_{s^-}) \right]dsd\zeta\right)}\\
&=: \int_0^t\left( \frac{g(Y_s)}{Y_s} + 2r\E \left[  \ln \Theta(Y_{s^-}) \right] \right)ds + \mathscr{M}_t, 
\end{align*}
where $\mathscr{M}$ is a martingale, as by assumption it has a finite variance.
To be more precise, we have
$$ Var (\mathscr{M}_t)= 2r \int_0^t \int_0^1 \E\left[   \ln^2 \Theta(Y_{s^-},\zeta) \right]dsd\zeta \leq 2r \sup_{x \geq 0} \E\left[   \ln^2 \Theta(x) \right] t= Ct, $$
where $C$ is a finite constant under the assumptions of point $i)$. Hence for $\eps>0$,
$$ \lim_{t \to \infty} \{\eps t + \mathscr{M}_t \} = \infty,\text{ almost surely,} $$
which implies
\begin{equation}\label{prod_pos} \Lambda_t - (\eta-\eps)t \geq \eps t + \mathscr{M}_t \underset{t \to \infty}{\to} \infty\text{ almost surely}. \end{equation}
We thus get
$$ \liminf_{t \to \infty} \P_x\left(Y_t e^{-(\eta-\eps)t}>0 \right)\geq \P_x\left(\liminf_{t \to \infty}Y_t e^{-(\eta-\eps)t}>0 \right)=  \P_x\left(\liminf_{t \to \infty}Y_te^{-\Lambda_t}e^{\Lambda_{t}-(\eta-\eps)t}>0 \right)>0, $$
where we used Fatou's Lemma, \eqref{Wpos} and \eqref{prod_pos}.
Hence, using \eqref{eq:mto} we obtain
$$\liminf_{t \to \infty} \E_{\delta_x}\left[ \frac{\sum_{u \in V_t} \mathbf{1}_{\{X_t^u e^{-(\eta-\eps)t}>0 \}} }{e^{(r-q) t}}\right]>0. $$
Now notice that the Cauchy-Schwarz inequality yields
\begin{align*}
 \E^2_{\delta_x}\left[  \frac{\sum_{u \in V_t} \mathbf{1}_{\{X_t^u e^{-(\eta-\eps)t}>0 \}} }{e^{(r-q) t}}\right]&\leq 
 \E_{\delta_x}\left[ \mathbf{1}_{\{N_t \geq 1\}}\left(\frac{\sum_{u \in V_t} \mathbf{1}_{\{X_t^u e^{-(\eta-\eps)t}>0 \}} }{N_t}\right)^2\right]
 \E_{\delta_x}\left[\left(\frac{N_t}{e^{(r-q) t}}\right)^2\right]\\
&  \leq \E_{\delta_x}\left[ \mathbf{1}_{\{N_t \geq 1\}}\frac{\sum_{u \in V_t} \mathbf{1}_{\{X_t^u e^{-(\eta-\eps)t}>0 \}} }{N_t}\right]
 \E_{\delta_x}\left[ \left(\frac{N_t}{e^{(r-q) t}}\right)^2\right],
\end{align*}
where the last inequality comes from the fact that the term in the first expectation in the right-hand side is smaller 
than one. 
The last expectation converges to $C:=1+(r+q)/(r-q)$ as $t$ goes to infinity (see Lemma 5.3 in \cite{marguet2022spread} in the case $\alpha=0$). Hence we get
\begin{align*}
0<C^{-1}\liminf_{t \to \infty} \E^2_{\delta_x}\left[ \frac{\sum_{u \in V_t} \mathbf{1}_{\{X_t^u e^{-(\eta-\eps)t}>0 \}} }{e^{(r-q) t}}\right]\leq 
\liminf_{t \to \infty} \E_{\delta_x}\left[\mathbf{1}_{\{N_t \geq 1\}} \frac{\sum_{u \in V_t} \mathbf{1}_{\{X_t^u e^{-(\eta-\eps)t}>0 \}} }{N_t}\right],
\end{align*}
and it ends the proof of point \ref{it:1-31}.\end{proof}

\appendix

\section{Existence and unicity of the host-parasite measure-valued process} \label{append_unic_exi}

\cha{This section is dedicated to the construction of the host-parasite measure-valued process $Z$ as the unique strong solution of a SDE.}

\cha{Recall the notation introduced in Section \ref{host-para-Z} and} 
let $\left(\Phi^u(x,s,t),u\in\mathcal{U},x\in\overline{\mathbb{\R}}_+, s\leq t\right)$ be a family of independent stochastic flows satisfying \eqref{X_sans_sauts} describing the individual-based dynamics.
Let $E = \mathcal{U}\times(0,1)\times \overline{\mathbb{R}}_+$  and 
$M(ds,du,d\zeta,dz)$ be a PPM on $\mathbb{R}_+\times E$ with intensity 
$ds\otimes n(du)\otimes d\zeta\otimes dz$, where $n(du)$ denotes the counting measure on $\mathcal{U}$. 
We assume that $M$ and $\left(\Phi^u,u\in\mathcal{U}\right)$ are independent. We denote by $\mathcal{F}_t$ the filtration generated by the restriction of the PPM $M$ to $[0,t]\times E$ and 
the family of stochastic processes $(\Phi^u(x,s, t), u\in\mathcal{U},x\in\overline{\mathbb{R}}_+,s\leq t)$ up to time $t$. \\

We now introduce assumptions to ensure the strong existence and uniqueness of the process. They are 
simpler than those of the companion paper \cite{marguet2022spread} because the \al{cell division rate} does not depend on the number of parasites they carry.

\begin{customass}{\bf{EU}}\label{ass_A}
\begin{enumerate}[label=\roman*)]
\item The function $p$ is locally Lipschitz on $\R_+$, non-decreasing and $p(0)= 0$. The function $g$ is continuous on $\R_+$, $g(0)=0$ and for any $n \in \N$ there exists a finite constant $B_n$ such that for any $0 \leq x \leq y \leq n$
\begin{align*} |g(y)-g(x)|
\leq B_n \phi(y-x),
\ 
\text{ where }\ 
\phi(x) = \left\lbrace\begin{array}{ll}
x \left(1-\ln x\right) & \textrm{if } x\leq 1,\\
1 & \textrm{if } x>1.\\
\end{array}\right.
\end{align*}
\item The function $\sigma$ is H\"older continuous with index $1/2$ on compact sets and $\sigma(0)=0$. 
\item The measure $\pi$ satisfies
$$ \int_0^\infty \left(z \wedge z^2\right)\pi(dz)<\infty. $$
\end{enumerate}
\end{customass}

Recall the definition of $\mathcal{G}$ in \eqref{def_gene}.
Then, the structured population process may be defined as the strong solution to a SDE. 

\begin{prop} \label{pro_exi_uni}
Under Assumption \ref{ass_A}, there exists a strongly unique $\mathcal{F}_t$-adapted 
c\`adl\`ag process $(Z_t,t\geq 0)$ taking values in $\mathcal{M}_P(\overline{\mathbb{R}}_+)$ such that for all $f\in C_0^2(\overline{\mathbb{R}}_+)$ and $x_0,t\geq 0$, 
\begin{align*}\nonumber
\langle Z_{t},f\rangle = &f\left(x_{0}\right)+\int_{0}^{t}\int_{\mathbb{R}_+}\mathcal{G}f(x)Z_{s}\left(dx\right)ds+M_{t}^f(x_0)&\\\nonumber
 +\int_{0}^{t}\int_{E}&\mathbf{1}_{\left\{u\in V_{s^{-}}\right\}}\left(\mathbf{1}_{\left\{z\leq r\right\}}\left(f\left(\Theta(X_{s^-}^u, \zeta) X_{s^-}^u \right)+ f\left((1-\Theta(X_{s^-}^u, \zeta)) X_{s^-}^u \right)-
f\left(X_{s^{-}}^{u}\right)\right)\right.\\
& \left.-\mathbf{1}_{\left\{0<z-r\leq q(X_{s^{-}}^u)\right\}}
f\left(X_{s^{-}}^{u}\right)\right) M\left(ds,du,d\zeta,dz\right),
\end{align*}
where for all $x\geq 0$, $M_{t}^f(x)$ is a $\mathcal{F}_t$-martingale.
\end{prop}

The proof is a combination of \cite[Proposition 1]{palau2018branching} and \cite[Theorem 2.1]{marguet2016uniform} (see  \cite[Appendix A]{marguet2022spread} for details).

\section{Technical lemma for the proof of Proposition \ref{prop:beta_ct:g_var}\ref{it:3-31}} \label{App_lemtechnq}

This appendix is dedicated to the statement and proof of a lemma, which is a slightly weaker version of \cite[Lemma 4.3]{BT11}. The only difference
is that they considered a Yule process instead of a birth and death process, and that the finite sets $I$ and $J$ could be arbitrary, whereas we 
impose the condition $J \subset I$.
The statement and proof are deliberately very close to that of \cite[Lemma 4.3]{BT11}. We give the proof in its entirety for the sake of readability.

\begin{lemma} \label{lemtechnq}Let $V$ be a denumerable subset and
 $(N_t(i) : t\geq 0)$  be i.i.d. birth and death processes with birth and death rates $r$ and $q< r$ for $i\in V$. Then there
exist $\delta>0$ and a nonnegative nonincreasing function  $G$  on $\R_+$
such that $G(y)\rightarrow 0$ as $y\rightarrow \infty$ and  for all $J\subset I$ finite subsets of $V$ and $x\geq 0$:
\begin{equation*}
\mathfrak{P}(x,\#J,\#I):=\P\Big(\sup_{t\geq 0} \mathbf{1}_{\{\sum_{i\in I} N_t(i)>0\}} \frac{\sum_{i\in J} N_t(i)}{\sum_{i\in I} N_t(i)} \geq x\Big)\leq G\left(\frac{\#I}{\#J}  x\right)+e^{-2\delta\#I}.
\end{equation*}
\end{lemma}
\begin{proof}
From classical results on birth and death processes (see \cite{athreya1972branching}
for instance), we know that for $i \in V$ the process $(N_t(i)e^{-(r-q)t})$ is a non 
negative martingale which converges to a random variable $W$ which is positive on the survival 
event (occurring with probability $(r-q)/r$).
Let us introduce the random variables,
$$M(i):=\sup_{t\geq 0} N_t(i)e^{-(r-q)t}\quad \mbox{ and } \quad m(i)=\inf_{t\geq 0} N_t(i)e^{-(r-q)t}.$$
$(M(i) : i \in V)$ and $(m(i): i \in V)$ are both sequences of finite nonnegative i.i.d.
random variables with finite expectation. 
Moreover, if we introduce, for $i \in V$, the events:
$$ V_\infty(i):= \{ N_t(i) \geq 1, \forall t\geq 0 \} \quad \text{and} \quad 
M_\infty(i):= \{ \exists t<\infty, N_t(i) = 0 \}, $$
and the set
$$ V_\infty:=\{ i \in I, V_\infty(i) \text{ holds}\}, $$
we have that
$ 0=m(i)\leq M(i) $ on the event $M_\infty(i)$, and $0<m(i)\leq M(i)$ on the event $V_\infty(i)$.
As a consequence, for any $\eps \in (0,1)$ and $t \geq 0$, using also that $J\subset I$, we have 
\begin{align*} \mathbf{1}_{\{\sum_{i\in I} N_t(i)>0\}} \frac{\sum_{i\in J} N_t(i)}{\sum_{i\in I} N_t(i)}&\leq  \mathbf{1}_{\{\#V_\infty>\eps \#I\}} \frac{\sum_{i\in J} M(i)}{\sum_{i\in I} m(i)}+\mathbf{1}_{\{\#V_\infty\leq \eps \#I\}} \\
&=  \frac{\#J}{\#I}\left(\mathbf{1}_{\{\#V_\infty>\eps \#I\}}\frac{\sum_{i\in J} M(i)}{\# J} \frac{\# I}{\sum_{i\in I} m(i)}+\mathbf{1}_{\{\#V_\infty\leq \eps \#I\}}\frac{\#I}{\#J}\right).
\end{align*}
Hence, we can bound $\mathfrak{P}$ as follows:
\begin{align}\nonumber\label{eq:app1}
\mathfrak{P}(x,\#J,\#I) \leq  & \P \left( \mathbf{1}_{\{\#V_\infty>\eps \#I\}}\frac{\sum_{i\in J} M(i)}{\# J} \frac{\# I}{\sum_{i\in I} m(i)} \geq \frac{\#I}{\#J}\frac{x}{2} \right)\\
&+ \P \left(\mathbf{1}_{\{\#V_\infty\leq \eps \#I\}}\frac{\#I}{\#J}\geq \frac{\#I}{\#J}\frac{x}{2}\right).
\end{align}

To handle the first term on the right-hand side of \eqref{eq:app1}, we define for $y\geq 0$
$$G(y)=\sup \Big\{  \P\Big(\mathbf{1}_{\{\#V_\infty>\eps \#I\}}\frac{\sum_{i\in J} M(i)}{\# J} \frac{\# I}{\sum_{i\in I} m(i)}\geq y\Big)  : J, I\subset V; \#I < \infty\Big\}.$$
By the law of large numbers, the sequence
$$\mathbf{1}_{\{\#V_\infty>\eps \#I\}}\frac{\sum_{i\in J} M(i)}{\# J} \frac{\# I}{\sum_{i\in I} m(i)}$$
is uniformly tight. So $G(y)\rightarrow  0$ as $y\rightarrow \infty$.

For the second term on the right-hand side of \eqref{eq:app1}, Markov's inequality yields
\begin{align*}
\P \left(\mathbf{1}_{\{\#V_\infty\leq \eps \#I\}}\frac{\#I}{\#J}\geq \frac{\#I}{\#J}\frac{x}{2}\right)\leq \P \left(\#V_\infty\leq \eps\#I\right) \frac{2}{x}.
\end{align*}
To bound the last term, we recall that $\# V_\infty$ is a sum of $\#I$ independent Bernoulli random variables with parameter $1-q/r$.
For $\eps \leq 1-q/r$, using Hoeffding's inequality, we obtain
\begin{align*}
\P \left(\#V_\infty\leq \eps\#I\right)\leq \exp\left(-2\#I\left(1-\frac{q}{r}-\eps\right)^2\right),
\end{align*}
and it concludes the proof.
\end{proof}

\section*{Acknowledgments}
The authors are grateful to V. Bansaye for his advice and comments, to B. Cloez for fruitful discussions, and to two anonymous referees for their suggestions, which helped improve the quality of the manuscript. 
This work was partially funded by the Chair "Mod\'elisation Math\'ematique et Biodiversit\'e" of VEOLIA-Ecole Polytechnique-MNHN-F.X. and by the French national research agency (ANR) via project ANR NOLO (ANR-20-CE40-0015) and by LabEx PERSYVAL-Lab (ANR-11-LABX-0025-01) funded by the French program Investissement d’avenir.

\bibliographystyle{abbrv}
\bibliography{biblio}

\end{document}